\documentclass{article}
\usepackage{graphicx}
\usepackage{amsmath}
\usepackage{amssymb}
\usepackage{amsthm}
\usepackage{mathrsfs}
\usepackage{ulem,xcolor}

\usepackage[all]{xy}

\usepackage{enumerate}
\usepackage{fullpage}

\newtheorem{theorem}{Theorem}[section]
\newtheorem{corollary}[theorem]{Corollary}

\newtheorem{proposition}[theorem]{Proposition}
\newtheorem{definition}[theorem]{Definition}
\theoremstyle{remark}
\newtheorem{rem}[theorem]{Remark}
\theoremstyle{definition}

\numberwithin{equation}{section}
\theoremstyle{theorem}

\theoremstyle{theorem}
\newtheorem*{mythma}{Theorem A}
\theoremstyle{theorem}
\newtheorem*{mythmb}{Theorem B}
\theoremstyle{theorem}
\newtheorem*{mythmc}{Theorem C}
\theoremstyle{theorem}

\theoremstyle{theorem}

\newcommand{\set}[1]{\left\{#1\right\}}

\newcommand{\R}{\mathbb R}
\newcommand{\Z}{\mathbb Z}
\newcommand{\N}{\mathbb N}

\newcommand{\C}{\mathbb C}

\newcommand{\HH}{\mathcal{H}}
\newcommand{\LL}{\mathcal{L}}
\newcommand{\PP}{\mathcal{P}}
\newcommand{\II}{\mathcal{I}}

\newcommand{\BBB}{\mathscr{B}}
\newcommand{\HHH}{\mathscr{H}}
\newcommand{\JJJ}{\mathscr{J}}
\newcommand{\LLL}{\mathscr{L}}
\newcommand{\MMM}{\mathscr{M}}

\newcommand{\PPP}{\mathrm{P}}

\newcommand{\EEE}{\mathbf{E}}
\newcommand{\FFF}{\mathbf{F}}

\newcommand{\eps}{\varepsilon}
\newcommand{\hyper}[2]{{}_{_{#1}}F_{_{#2}}}
\newcommand{\hyperbold}[2]{{}_{_{#1}}\mathbf{F}_{_{#2}}}

\numberwithin{equation}{section}

\begin{document}

\title{Series expansions for Maass forms on the full modular group from the Farey transfer operators}

\author{Claudio Bonanno\footnote{Dipartimento di Matematica,
Universit\`a di Pisa, Largo Bruno Pontecorvo 5, I-56127 Pisa, Italy. Email: claudio.bonanno@unipi.it}
\and Stefano Isola\footnote{Scuola di Scienze e Tecnologie, Universit\`a di Camerino, via Madonna delle Carceri, I-62032 Camerino, Italy. Email: stefano.isola@unicam.it}}
\date{}

\maketitle

\begin{abstract}
We deepen the study of the relations previously established by Mayer, Lewis and Zagier, and the authors, among the eigenfunctions of the transfer operators of the Gauss and the Farey maps, the solutions of the Lewis-Zagier three-term functional equation and the Maass forms on the modular surface $PSL(2,\Z)\backslash \HH$. In particular we introduce an ``inverse'' of the integral transform studied by Lewis and Zagier, and use it to obtain new series expansions for the Maass cusp forms and the non-holomorphic Eisenstein series restricted to the imaginary axis. As corollaries we obtain further information on the Fourier coefficients of the forms, including a new series expansion for the divisor function.
\end{abstract}

\section{Introduction} \label{intro}

One of the most interesting objects in the mathematics literature are the Maass forms on the full modular group $\mathrm{PSL}(2,\Z)$. Letting $\Delta := -y^{2} \left( \frac{\partial^{2}}{\partial x^{2}} + \frac{\partial^{2}}{\partial y^{2}}\right)$ denote the hyperbolic Laplacian, Maass forms are smooth $\mathrm{PSL}(2,\Z)$-invariant complex functions $\phi$ defined on the upper half-plane $\HH=\set{z=x+i y \, : \, y>0}$, increasing less than exponentially as $y\to \infty$, and satisfying $\Delta \phi = \lambda \phi$ for some $\lambda \in \C$. Maass forms divide into cusp and non-cusp forms according to their behaviour at the cusp of the modular surface $\mathrm{PSL}(2,\Z)\backslash\HH$, and into even and odd forms according to whether $\phi(-x+i y) = \pm \phi(x+i y)$.

\noindent
Despite their importance, Maass cusp forms remain mysterious objects. No explicit construction exists and all basic information about their existence comes from the Selberg trace formula. Much more is known for the non-cusp forms, which are generated by the non-holomorphic Eisenstein series. The standard approach to Maass forms uses the methods of harmonic analysis on $\HH$, which leads to the Fourier expansion of the forms in terms of Whittaker function (see e.g. \cite{iwaniec}).

\noindent
In recent years, a new approach to Maass forms has been developed using the relation between the Selberg zeta function $Z(q)$ and the Fredholm determinant of the transfer operators $\LL_{q}$ of the Gauss map (see \cite{Ma3,CM1}), also in connection with a functional approach introduced in \cite{Le,LeZa}. This connection becomes clearer if one considers the transfer operators $\PP_q$ of the Farey map, a ``slow'' version of the Gauss map, as shown by the authors in \cite{BI}, where we also studied the properties of the eigenfunctions of the operators $\PP_q$. In this paper we continue the work set out in \cite{BI}, by transferring the information on the eigenfunctions of $\PP_q$ to Maass forms. In particular we use the integral transform studied in \cite{LeZa} to obtain series expansions for the Maass forms restricted to the imaginary axis that to our knowledge are entirely new. We first obtain expansions in terms of Legendre functions $\PPP_{\nu}^{\mu}$. In Section \ref{sec-u} we prove the following
\begin{mythma}
If $u(x+iy)$ is an even Maass cusp form on $\mathrm{PSL}(2,\Z)$ with eigenvalue $q(1-q)$, then there exists a sequence $\{a_{n,q}\}$ satisfying $\limsup_{n} |a_{n,q}|^{\frac 1n} \le 1$, such that 
\[
u(iy) = \sum_{n=0}^{\infty}\, (-1)^{n}\, a_{n,q}\, \frac{y^{\frac 12}\, \PPP^{-q+\frac 12}_{n+q-\frac 12} \left( \frac{1}{(1+y^{2})^{\frac 12}} \right) + y^{n+q}\, \PPP^{-q+\frac 12}_{n+q-\frac 12} \left( \frac{y}{(1+y^{2})^{\frac 12}} \right)}{(1+y^{2})^{\frac n2 + \frac q2 + \frac 14}}
\]
uniformly in $y$ on any compact interval in $(0,+\infty)$. 

\vskip 0.1cm
\noindent
The non-holomorphic Eisenstein series $E(x+iy,q)$ can be written for $x=0$ as a meromorphic function on the half-plane $\Re(q)>0$ as
\[
\begin{aligned}
E(iy,q) = \, & \zeta(2q)\, \Big( y^{q} + y^{-q} \Big) - 2\, \zeta(2q)\, \left( \frac{y}{1+y^2} \right)^q +\\[0.2cm]
& + 2^{q+\frac 12}\, \Gamma\left(q+\frac 12\right) \sum_{n=0}^{\infty}\, (-1)^{n}\, b_{n,q}\, \frac{y^{\frac 12}\, \PPP^{-q+\frac 12}_{n+q-\frac 12} \left( \frac{1}{(1+y^{2})^{\frac 12}} \right) + y^{n+q}\, \PPP^{-q+\frac 12}_{n+q-\frac 12} \left( \frac{y}{(1+y^{2})^{\frac 12}} \right)}{(1+y^{2})^{\frac n2 + \frac q2 + \frac 14}} 
\end{aligned}
\]
with
\[
b_{n,q} := (-1)^{n}\, \frac{\Gamma(n+2q)}{(n+1)!\, \Gamma(2q)}\, \sum_{i=0}^{n}\, {n+1\choose i}\, B_{i}\, \zeta(2q-1+i)\, ,
\]
where $\{B_i\}$ are the Bernoulli numbers and $\zeta(s)$ is the Riemann zeta function.
\end{mythma}
\noindent
An analogous result is given in Theorem \ref{ext-to-0-odd} for odd Maass cusp forms. Moreover, in Appendix B we show a curious way of using the Legendre functions to expand $y^q$, and then to write the non-holomorphic Eisenstein series in terms of the Legendre functions.

\noindent
Then in Section \ref{sec-courier} we study the Fourier coefficients of Maass forms $-$ which for the non-cusp case are related to the divisor function $\sigma_{\ell}(n)$ $-$ and obtain some results which can be summarized in the following
\begin{mythmb}
Let $\{c_{n,q}\}$ denote the coefficients of the Fourier expansion of an even Maass cusp form with eigenvalue $q(1-q)$. Then we have, up to a constant depending on $q$, 
\[
c_{n,q} = 2\, n^{q-\frac 12}\sum_{k=1}^{\infty}\, \frac{(-1)^{k} a_{2k,q} }{\Gamma(2k+2q)} (2\pi n)^{2k}
\] 
where $\{a_{n,q}\}$ is the sequence introduced in Theorem A. 
\vskip 0.1cm
\noindent
In the case of non-cusp forms we prove that for $n\ge 1$ and $\Re(q)>0$ it holds
\[
\sigma_{2q-1}(n) = n^{2q-1}\, \sum_{k=1}^{\infty}\, \frac{(-1)^k\, \tilde A_{2k,q}}{(2k)!}\, (2\pi n)^{2k}
\]
with
\[
\tilde A_{k,q} :=  \frac{1}{k+1}\, \sum_{i=2}^{k}\, {k+1\choose i}\, B_{i}\, \zeta(2q-1+i)\, .
\]
\end{mythmb}
\noindent
In Remark \ref{rem-ram} we show that the series expansion for the divisor function can be considered an extension of the Ramanujan expansion.

\noindent
Finally, in Section \ref{sec-power} we exploit the properties of the Legendre function to obtain new series expansions for the Maass forms. In the cusp case these are only formal since we don't have control on the coefficients, whereas in the non-cusp case we prove 
\begin{mythmc}
For $q$ with $\Re(q)>0$ it holds
\[
\begin{aligned}
E(iy,q) = &\,  2 \, \left(\frac{y}{1+y^2}\right)^q \, \left[ \zeta(2q-1)\, \hyper{2}{1}\left(1,q\, ;\, \frac 32\, ;\, \frac{1}{1+y^2}\right) + \zeta(2q-1)\, \hyper{2}{1}\left(1,q\, ;\, \frac 32\, ;\, \frac{y^2}{1+y^2}\right) \right] + \\[0.3cm]
& + 4\, \left(\frac{y}{1+y^2}\right)^q \, \sum_{s=1}^\infty\, 2^s\, \frac{\Gamma(s+q)}{s!\, \Gamma(q)} \left( \sum_{k=1}^s\, {s\choose k-1}\, \frac{(-1)^k}{2^k\, (s+k)}\, B_{s+k}\, \zeta(2q-1+s+k) \right) \, \frac{1+y^{2s}}{(1+y^{2})^{s}}
\end{aligned}
\]
uniformly in $y$ on any compact interval in $(0,+\infty)$.
\end{mythmc}

\noindent We believe that these new series expansions will turn useful in the study of Maass forms and their Fourier coefficients, as they involve the coefficients $\{a_{n,q}\}$ which come from the totally different approach described below. In particular we hope that this will stimulate new numerical investigations of the coefficients $\{a_{n,q}\}$ (see also Remark \ref{rem-ref}).

For the benefit of the reader we now briefly recall the main steps of the approach to the Maass forms as developed by Mayer, Lewis, Zagier and the authors in \cite{Ma3, LeZa, BI}.

\noindent
In \cite{Ma3} Mayer used the definition of the Selberg zeta function $Z(q)$ as a product over the length spectrum of $\mathrm{PSL}(2,\Z)\backslash\HH$ to prove a relation between $Z(q)$ and the Smale-Ruelle zeta function for the geodesic flow on the modular surface. We recall that the length spectrum of $\mathrm{PSL}(2,\Z)\backslash\HH$ is the set of lengths of the closed geodesics on $\mathrm{PSL}(2,\Z)\backslash\HH$, and the closed geodesics appear in the definition of the Smale-Ruelle zeta function. The aforementioned relation together with results in \cite{Ma2} entails the main result of \cite{Ma3}, the equality
\begin{equation}\label{fund-mayer}
Z(q) = \det (1-\LL_{q})\, \det (1+\LL_{q})\, , \qquad q\in \C\, .
\end{equation}
Here ``$\det$'' indicates the determinant in the sense of Fredholm, and $\LL_{q}$ denotes the meromorphic extension to $q\in \C$ of the family of nuclear of order zero endomorphisms defined by
$$
\left(\LL_{q}h\right) (z) = \sum_{n=1}^{\infty}\, \frac{1}{(z+n)^{2q}}\, h\left(\frac{1}{z+n}\right)
$$
for $\Re(q)> \frac 12$, on the space $H(D)$ of holomorphic functions in the disk $D= \set{z\in \C \, :\, |z-1| < \frac 32}$. The connection \eqref{fund-mayer} comes from the arithmetic properties of the length spectrum of $\mathrm{PSL}(2,\Z)\backslash\HH$ and the fact that the endomorphisms $\LL_{q}$ are the transfer operators of the Gauss map, which generates a dynamical system related to the continued fractions expansion of a real number. Combining Mayer's equality \eqref{fund-mayer} sharpened by Efrat in \cite{efrat} with the known positions of the zeroes of $Z(q)$ as implied by the Selberg trace formula, one can state the following
\begin{theorem}[\cite{efrat},\cite{Ma3}] \label{thm-mayer}
Let $q=\xi +i \eta$ be a complex number with $\xi>0$ and $q\not=\frac 12$. Then:
\begin{enumerate}[(i)]
\item there exists a nonzero $h\in H(D)$ such that $\LL_{q}h = h$ if and only if $q$ is either an even spectral parameter of $\Gamma$, that is there exists an even Maass cusp form $u$ such that $\Delta u = q(1-q) u$, or $2q$ is a non-trivial zero of the Riemann zeta function, or $q=1$;
\item there exists a nonzero $h\in H(D)$ such that $\LL_{q}h = -h$ if and only if $q$ is an odd spectral parameter of $\Gamma$, that is there exists an odd Maass cusp form $u$ such that $\Delta u = q(1-q) u$.
\end{enumerate}
\end{theorem}

\noindent
In the papers \cite{Le,LeZa} Lewis and Zagier introduced a three-term functional equation whose solutions are in one-to-one correspondence with the Maass cusp and non-cusp forms. Using the results for the spectral parameters of $\mathrm{PSL}(2,\Z)\backslash\HH$ and for the Maass non-cusp forms, they proved the following result.
\begin{theorem}[\cite{LeZa}] \label{thm-lz}
There is an isomorphism between the Maass cusp forms with eigenvalue $q(1-q)$ and the space of real-analytic solutions of the three-term functional equation
\begin{equation}\label{fund-lz}
\psi(x) = \psi(x+1) + (x+1)^{-2q} \psi\left( \frac{x}{x+1} \right)\, , \qquad x\in \R^{+}
\end{equation}
with the conditions
\begin{equation}\label{fund-lz-cond}
\psi(x)= O(1) \ \text{as }\, x\to 0^{+}\, , \quad \psi(x)= O(1/x) \ \text{as }\, x\to +\infty
\end{equation}
Moreover, the Maass non-cusp forms, which for any given $q$ lie in a one-dimensional space generated by the non-holomorphic Eisenstein series $E(z,q)$, are in one-to-one correspondence with the functions
\begin{equation}\label{zagier-psi}
\psi_{q}^{+}(x) = \frac{\zeta(2q)}{2} \left( 1 + x^{-2q} \right) + \sum_{m,n\ge 1} \frac{1}{(mx+n)^{2q}}\, , \qquad \Re(q)>1
\end{equation}
which, when multiplied by $\frac{\Gamma(2q)}{\Gamma(q-1)}$, can be analytically continued to $q\in \C$ as solutions of \eqref{fund-lz}.
\end{theorem}

\noindent
In \cite{LeZa} the solutions of equation \eqref{fund-lz} are called \textit{period functions} because of an analogy, explored in the paper, with the classical Eichler-Shimura-Manin period polynomials of the holomorphic cusp forms. Moreover, the period functions associated to a Maass forms are divided into even and odd functions.

\noindent
Putting together Theorems \ref{thm-mayer} and \ref{thm-lz} we have the following situation for the zeroes of the Selberg zeta function $Z(q)$:
\begin{itemize}
\item if $q$ is an even spectral parameter with $\xi= \frac 12$, then there exist a nonzero $h\in H(D)$ such that $\LL_{q}h = h$ and an even real-analytic function $\psi(x)$ which satisfies \eqref{fund-lz} with conditions \eqref{fund-lz-cond};

\item if $q$ is an odd spectral parameter with $\xi= \frac 12$, then there exist a nonzero $h\in H(D)$ such that $\LL_{q}h = -h$ and an odd real-analytic function $\psi(x)$ which satisfies \eqref{fund-lz} with conditions \eqref{fund-lz-cond};

\item if $2q$ is a non-trivial zero of the Riemann zeta function, then there exist a nonzero $h\in H(D)$ such that $\LL_{q}h = h$ and \eqref{fund-lz} has solutions given by multiples of the analytic continuation of the function $\psi_{q}^{+}$;

\item if $q=1$ then there exist a nonzero $h\in H(D)$ such that $\LL_{q}h = h$, in fact we have $h(x)=\frac{1}{x+1}$, and \eqref{fund-lz} has solutions given by multiples of the function $\psi_{1}^{+}(x) = \frac 1x$.
\end{itemize}
Moreover there is an explicit relation between the eigenfunctions of the operator $\LL_{q}$ and the period functions relative to the same $q$. Namely $h(x)=\psi(x+1)$, and the same holds on $D$, where $\psi(z+1)$ is the holomorphic extension of $\psi$ to $\C\setminus (-\infty,0]$.

\noindent
The beauty of Mayer's result lies in the displaying of the power of the theory of transfer operators for dynamical systems, but the spectral properties of the operators $\LL_{q}$ turned out to be difficult to study, see \cite{CM1} and \cite{BJ}. On the other side, Lewis and Zagier approach has the advantage of introducing a relation of Maass forms with solutions of an equation with a finite number of terms, which might be easier to handle.

\noindent
These two aspects are combined in our paper \cite{BI}, where we used a family of signed transfer operators $\PP_{q}^{\pm}$ for the Farey map, a ``slow'' version of the Gauss map, defined for $\xi = \Re(q)>0$ by
\[
(\PP_{q}^{\pm}f)(z) = (\PP_{0,q}f)(z)\pm (\PP_{1,q}f)(z) := \left( \frac{1}{z+1} \right)^{2q} f\left( \frac{z}{z+1} \right) \pm \left( \frac{1}{z+1} \right)^{2q} f\left( \frac{1}{z+1} \right)
\]
where $z\in B = \set{z\in \C\ :\ |z-\frac 12| < \frac 12}$ and $f\in H(B)$. We studied the problem of existence of eigenfunctions for $\PP_{q}^{\pm}$ and proved the following
\begin{theorem}[\cite{BI}]\label{uno-bi}
(a) If $f\in H(B)$ satisfies $\PP^{+}_{q}f=\lambda f$ with $\lambda \not= 0$ then $f\in H(\set{\Re(z)>0})$ and we call it even in the sense that $\II_{q}f = f$, where
\begin{equation}\label{invo}
(\II_{q}f)(z):= \frac{1}{z^{2q}}\, f\left( \frac 1z \right)
\end{equation}
Moreover it satisfies
\begin{equation}\label{gen-lz}
\lambda f(z) = f(z+1) + (z+1)^{-2q} f\left( \frac{z}{z+1} \right)\, , \qquad \Re(z)>0\, .
\end{equation}
(b) If $f\in H(B)$ satisfies $\PP^{-}_{q}f=\lambda f$ with $\lambda \not= 0$ then $f\in H(\set{\Re(z)>0})$ and we call it odd in the sense that $\II_{q}f = - f$. Moreover it satisfies \eqref{gen-lz}.\\[0.2cm]
(c) If $f\in H(\set{\Re(z)>0})$ satisfies \eqref{gen-lz} for $\lambda\not = 0$, then $\PP_{q}^{\pm}(f \pm \II_{q}f) = \lambda (f \pm \II_{q}f)$.\\[0.2cm]
(d) If $f\in H(B)$ satisfies $\PP^{+}_{q}f=\lambda f$ with $\lambda \not\in [0,1)$ then there exists $\phi\in L^{2}((0,+\infty),t^{2\xi -1}e^{-t}dt)$ such that $f$ can be written as
\begin{equation}\label{forma-eig-p}
f(z) = c\, \frac{\lambda^{\frac 1z}}{z^{2q}} + b\, \frac{\Gamma(2q-1)}{\Gamma(2q)} \, \frac 1z + \frac{1}{z^{2q}}\, \int_{0}^{\infty}\, e^{-\frac tz}\, \phi(t)\, t^{2q-1}\, dt\, , \qquad \Re(z)>0\, ,
\end{equation}
where $c,b \in \C$, $\phi(0)$ is finite and $\phi(t)-\phi(0) = O(t)$ as $t\to 0^{+}$, and the last term is bounded as $\Re(z) \to 0$. Moreover if $\lambda\not =1$ then $b=0$.\\[0.2cm]
(e) If $f\in H(B)$ satisfies $\PP^{-}_{q}f=\lambda f$ with $\lambda \not\in [0,1)$ then there exists $\phi\in L^{2}((0,+\infty),t^{2\xi -1}e^{-t}dt)$ such that $f$ can be written as
\begin{equation}\label{forma-eig-m}
f(z) = c\, \frac{\lambda^{\frac 1z}}{z^{2q}} + \frac{1}{z^{2q}}\, \int_{0}^{\infty}\, e^{-\frac tz}\, \phi(t)\, t^{2q-1}\, dt\, , \qquad \Re(z)>0\, ,
\end{equation}
where $c \in \C$, $\phi(0)$ is finite and $\phi(t)-\phi(0) = O(t)$ as $t\to 0^{+}$, and the last term is bounded as $\Re(z) \to 0$.
\end{theorem}
\noindent
Using the operators $\PP_{q}^{\pm}$ we introduced a generalization of the transfer operators $\LL_{q}$, namely the  two variable operator-valued function  $\LL_{q,w}$ formally defined as
$$
\LL_{q,w} = w\, \PP_{1,q} (1-w\, \PP_{0,q})^{-1}
$$
We proved that as operators acting on the Banach space $H_{\infty}(D_{\eps})$ of functions holomorphic on $D_{\eps}=\set{z\in \C\ :\ |z-1| < \frac 32-\eps}$ and bounded on $\overline{D_{\eps}}$, they are nuclear of order zero for $\Re(q)>0$ and $w\in \C\setminus (1,\infty)$. Moreover the function $q\mapsto \LL_{q,w}$ is analytic in $\Re(q)>0$ for any $w\in \C\setminus [1,\infty)$ and is meromorphic in $\Re(q)>0$ for $w=1$ with a simple pole at $q=\frac 12$. Analogously the function $w\mapsto \LL_{q,w}$ is analytic in $w\in \C\setminus [1,\infty)$ for any $q$ with $\Re(q)>0$. Hence we can compute the Fredholm determinants of the operators $(1\pm \LL_{q,w})$ and define the \textit{two-variable Selberg zeta function}
\[
Z(q,w) := \det (1-\LL_{q,w})\, \det (1+\LL_{q,w})
\]
for $\Re(q)>0$ and $w\in \C\setminus (1,\infty)$. For $w=1$ the function $Z(q,1)$ is meromorphic in $\Re(q)>0$ with a simple pole at $q=\frac 12$ and coincides with the Selberg zeta function $Z(q)$.

\noindent
Finally we obtained a relation between the eigenfunctions of $\LL_{q,w}$ and those of $\PP_{q}^{\pm}$, and therefore, thanks to Theorem \ref{uno-bi}-(a,b),  a relation between the solutions of the generalized three-term functional equation \eqref{gen-lz} and the zeroes of the function $Z(q,w)$. More precisely, using \cite[Theorem 3.6]{BI} and \cite[Corollary 3.7]{BI}, and the definition of $Z(q,w)$, together with the spectral characterisation of the zeroes of the Selberg zeta function given in Theorem \ref{thm-mayer}, it follows
\begin{theorem}[\cite{BI}]\label{due-bi}
(a) Let $w=1$. Then:
\begin{itemize}
\item $q$ is an even spectral parameter with $\xi= \frac 12$ if and only if there exists an even $f\in H(B)$ such that $\PP_{q}^{+}f = f$, $f$ satisfies \eqref{gen-lz} with $\lambda=1$ (or \eqref{fund-lz}) and it can be written as in \eqref{forma-eig-p} with $c=b=0$;

\item $q$ is an odd spectral parameter with $\xi= \frac 12$ if and only if there exists an odd $f\in H(B)$ such that $\PP_{q}^{-}f = f$, $f$ satisfies \eqref{gen-lz} with $\lambda=1$ (or \eqref{fund-lz}) and it can be written as in \eqref{forma-eig-m} with $c=0$;

\item $2q$ is a non-trivial zero of the Riemann zeta function if and only if there exists an even $f\in H(B)$ such that $\PP_{q}^{+}f = f$, $f$ satisfies \eqref{gen-lz} with $\lambda=1$ (or \eqref{fund-lz}) and it can be written as in \eqref{forma-eig-p} with $c=0$ and $b\not= 0$;

\item $q=1$ is a zero of $Z(q,1)$ since $f(z) = \frac 1z$ satisfies $\PP_{1}^{+}f = f$.
\end{itemize}
(b) Let $w\in \C\setminus [1,\infty)$. Then:
\begin{itemize}
\item $q$ is an ``even'' zero of $Z(q,w)$ if and only if there exists an even $f\in H(B)$ such that $\PP_{q}^{+}f = \frac 1w f$, $f$ satisfies \eqref{gen-lz} with $\lambda=\frac 1w$ and it can be written as in \eqref{forma-eig-p} with $c=b=0$;

\item $q$ is an ``odd'' zero of $Z(q,w)$ if and only if there exists an odd $f\in H(B)$ such that $\PP_{q}^{-}f = \frac 1w f$, $f$ satisfies \eqref{gen-lz} with $\lambda=\frac 1w$ and it can be written as in \eqref{forma-eig-m} with $c=0$.
\end{itemize}
\end{theorem}

\noindent
Since by Theorem \ref{uno-bi}-(a,b), eigenfunctions of $\PP_{q}^{\pm}$ satisfy a three-term equation which is a generalization of the Lewis-Zagier equation \eqref{fund-lz}, we call the functions $f$ of Theorem \ref{due-bi} \textit{generalized period functions (gpf)} associated to the zeroes of the zeta function $Z(q,w)$, \textit{even} and \textit{odd} according to whether they correspond to even or odd zeroes. In addition we distinguish the two classes of gpf with $b=0$, which we call $0$-gpf, and $b\not= 0$, which we call $b$-gpf. In the $w=1$ case the $0$-gpf correspond to the Maass cusp forms and the $b$-gpf to the non-cusp forms. On the contrary in the $w\not= 1$ case the set of $b$-gpf is empty.

\noindent
In Section \ref{sec-u} we consider the general case $w\in \C\setminus (1,\infty)$, so we find a series expansion as in Theorem A also for functions $u_w$ corresponding to $0$-gpf with $w\not= 1$. However it is not clear if these functions play a role in the spectral theory of hyperbolic surfaces. The original aim of this research was exactly to find a characterization for general $u_w$, and even if we don't achieve this result in this paper, we believe this is an interesting problem to be studied.

Finally we would like to point out that this paper includes one possible extension of the works \cite{LeZa,Ma3}. Other directions can be found in \cite{DH,CM2,MMS}, where the authors study the relation between period functions and Maass wave forms for subgroups of $\mathrm{PSL}(2,\Z)$, and in \cite{MP,pohl}, where the role of the ``slow'' dynamics and its advantage of introducing a transfer operator with finitely many terms is studied in relation to the cohomological approach of \cite{BLZ}.

\section{Notations for special functions and integral transforms} \label{sec:sfit}

We use standard notations: $\hyper{2}{1}(a,b;c;x)$ for the hypergeometric function; $J_{\nu}(z)$ for the Bessel functions of first kind; $K_{\nu}(z)$ for the modified Bessel functions of the third kind; $L_{n}^{\nu}(t)$ for the generalized Laguerre polynomials; $\Gamma(\nu)$ for the Gamma function; $\zeta(q)$ for the Riemann zeta function; $\PPP_{\nu}^{\mu}$ for the Legendre functions in the real interval $(-1,1)$.

\noindent
In the following we use the following integral transforms:
\begin{itemize}
\item \textit{Laplace transform} $$\LLL[\varphi](z) := \int_{0}^{\infty}\, e^{-zt}\, \varphi(t)\, dt$$

\item \textit{Symmetric Hankel transform} $$\HHH_{\nu}[\varphi](z) := \int_{0}^{\infty}\, J_{\nu}(tz)\, \sqrt{tz}\ \varphi(t)\, dt$$

\item \textit{Asymmetric Hankel transform} $$\JJJ_{\nu}[\varphi](z) := \int_{0}^{\infty}\, J_{2\nu}(2\sqrt{tz})\, \left( \frac tz \right)^{\nu}\, \varphi(t)\, dt$$

\item \textit{Borel generalized transform} $${\BBB}_\nu [\varphi](z):= \frac{1}{z^{2\nu}}\ \int_0^\infty e^{-\frac{t}{z}}\, t^{2\nu -1} \varphi (t)\, dt$$

\item \textit{Mellin transform} $${\MMM} [\varphi](\rho):= \int_0^\infty \varphi(t)\, t^{\rho -1} \, dt$$
\end{itemize}
The asymmetric Hankel transform has been introduced in \cite{Le}, and the Borel generalized transform in \cite{Is}. For the other transforms see \cite{E2} and \cite{GR}. For the convergence of the Hankel transforms, we recall that  the Bessel function $J_\nu(t)$ satisfies the estimates $J_\nu(t) = O(t^\nu)$ as $t\to 0^+$, and $J_\nu(t) = O(t^{-\frac 12})$ as $t\to \infty$ (see \cite[vol. II]{E}).

\noindent
We also use the notation
\[
\chi_{\alpha}(t) := t^{\alpha}\quad \text{and} \quad \exp_{\alpha}(t):= e^{\alpha t}\, , \quad \alpha \in \C
\]
and write $q= \xi + i \eta$, with $\xi>0$ and $\eta \in \R$. Moreover we write $f(z) \doteq g(z)$ for two functions $f,g$ which coincides up to a non-vanishing multiplication constant possibly depending only on $q$.

\section{From gpf to Maass forms on the imaginary axis} \label{sec-u}

To study the set of gpf, we used in \cite{BI} the integral transform $\BBB_{q}$ on the spaces of functions $L^p(m_q)$ in $\R^{+}$ with $m_q(dt)= t^{2\xi-1}\, e^{-t}\, dt$. Letting
\[
L^{p}(m_{q}) := \set{\phi :\R^{+}\to \C\ :\ \int_{0}^{\infty}\, |\phi(t)|^{p}\, t^{2\xi -1} e^{-t}\, dt <\infty}
\]
with the norm
$$
\| \phi \|_{p} := \left( \int_{0}^{\infty}\, |\phi(t)|^{p}\, t^{2\xi -1} e^{-t}\, dt \right)^{\frac 1p}\, ,
$$
it is immediate to check that
$$
L^{1}(m_{q}) \ni \phi \mapsto \BBB_{q}[\phi] \in H(B)
$$
and that $\BBB_{q}$ is continuous on $L^{1}(m_{q})$ with values on $H(B)$ with the standard topology induced by the family of supremum norms on compact subsets of $B$. 
Moreover, since $m_{q}(0,\infty) = \Gamma(2\xi)$, one has $L^{p}(m_{q}) \subset L^{1}(m_{q})$ for all $p\in [1,\infty]$. 

\noindent
We also need to introduce the linear operators $M$ and $N_q$ defined by
\[
M (\phi) (t) := e^{-t}\, \phi(t)
\]
\[
N_q (\phi) (t) := \JJJ_{q-\frac 12}[\exp_{-1}\phi](t) =  \int_0^\infty\ J_{2q-1}(2\sqrt{st}) \left(\frac st\right)^{q-\frac 12}\, e^{-s}\, \phi(s)\, ds\, .
\]
In \cite{BI} it is proved that
\[
L^{2}(m_{q}) \ni \phi \mapsto N_{q}[\phi] \in L^{2}(m_{q})\, .
\]
The same is clearly true also for $M$. Moreover in \cite[Proposition 2.5]{BI}, it is proved that the transfer operators of the Farey map $\PP_q^\pm$ has a particularly nice behaviour with respect to the Borel generalized transform. In particular for all $\phi \in L^2(m_q)$ it holds
\begin{equation}\label{trasf-op-l2}
\PP_q^\pm \Big( \BBB_q \left[ \chi_{-1} + \phi \right] \Big) (z) = \BBB_q \Big[ (M\pm N_q) \left( \chi_{-1} + \phi \right) \Big] (z)\, .
\end{equation}

\noindent
Finally, putting together Theorem 2.8 and Corollary 2.10 in \cite{BI}, we have
\begin{proposition}[\cite{BI}] \label{le-fi}
If $f$ is a generalized period function associated to a zero $q$ and to the eigenvalue $\lambda = \frac 1w$, then there exist $b\in \C$ and a function $\varphi \in L^{2}(m_{q})$ such that
\begin{equation}\label{eq-transf-f}
f(z) = \BBB_{q}\left[ \frac{b}{\Gamma(2q)}\, \chi_{-1} + \varphi \right](z) 
\end{equation}
and 
\[
(M\pm N_{q})\left( \frac{b}{\Gamma(2q)}\, \chi_{-1} + \varphi \right) = \lambda \left( \frac{b}{\Gamma(2q)}\, \chi_{-1} + \varphi \right)\, ,
\]
where the signs ``+'' and ``-'' correspond to the case of even or odd gpf respectively.
Moreover, there exists a sequence $\{ a_{n,q}\}_{n\ge 0}$ with $\limsup |a_{n,q}|^{\frac 1n}\le 1$ such that: in the even case, for $w=1$
\begin{equation}\label{fi-power-b}
\varphi(t) = \frac{e^{-t}}{1-e^{-t}} \, \sum_{n=1}^{\infty}\, \frac{(-1)^{n} a_{n,q}\, t^{n}}{\Gamma(n+2q)} + \frac{a_{0,q}}{\Gamma(2q)} \left( \frac{e^{-t}}{1-e^{-t}} - \frac 1t \right) 
\end{equation}
with $a_{0,q}=b$, and for $w\in \C\setminus [1,\infty)$,
\begin{equation}\label{fi-power-0}
\varphi(t) = \frac{we^{-t}}{1-we^{-t}} \, \sum_{n=0}^{\infty}\, \frac{(-1)^{n} a_{n,q}\, t^{n}}{\Gamma(n+2q)}\, ;
\end{equation}
in the odd case, for all $w\in \C \setminus (1,\infty)$,  the constant $b$ in \eqref{eq-transf-f} vanishes and the function $\varphi$ can be written as in \eqref{fi-power-0}.

\noindent
Finally, the invariance under the involution $\II_{q}$ defined in \eqref{invo}, implies that if $b=0$
\begin{equation}\label{invo-fi}
\BBB_{q}[\varphi] = \pm\,  \LLL[\chi_{2q-1}\varphi]
\end{equation}
where again the signs ``+'' and ``-'' correspond to the case of even or odd gpf respectively.
\end{proposition}

\subsection{The even case for $0$-gpf} \label{sec:even-good}

In \cite{Le} and \cite{LeZa} it is proved that the set of even period functions, that is even 0-gpf with $w=1$, is in one-to-one correspondence with the set of even Maass cusp forms. This correspondence is proved using the Fourier series expansions of the even cusp forms given by
\begin{equation}\label{four-exp}
u(x+i y) = y^{\frac 12}\, \sum_{n\ge 1}\, c_{n,q}\, K_{q-\frac 12}(2\pi ny)\, \cos(2\pi nx)
\end{equation}
where the coefficients $c_{n,q}$ have at most polynomial growth. In particular the correspondence is given in \cite{Le} in terms of the Laplace and Hankel transforms as
\begin{equation}\label{formula-1}
\psi(z) = \LLL\left[ \chi_q \HHH_{q-\frac 12} [u(i y)] \right] (z)\, .
\end{equation}
Since the gpf $f(z)$ of Proposition \ref{le-fi} coincides with $\psi(z)$ up to a multiplication constant, using \eqref{eq-transf-f} with $b=0$ and \eqref{invo-fi}, we obtain
$$ 
\LLL\left[ \chi_q \HHH_{q-\frac 12} [u(i y)] \right] (z) \doteq \LLL[\chi_{2q-1}\varphi](z)
$$
from which we obtain an integral correspondence between cusp forms and the eigenfunctions $\varphi$ of Proposition \ref{le-fi}, namely
\begin{equation}\label{u-to-fi}
\varphi(t) \doteq t^{1-q}\, \HHH_{q-\frac 12} [u(i y)](t)
\end{equation}
see \cite[equation (2.27)]{LeZa}. 

\begin{rem} \label{notazione-cost}
We have used the notation $\doteq$ here and in the following to denote an equality up to a multiplicative constant between cusp forms and the eigenfunctions $\varphi$. However, once this constant has been fixed it remains the same in all the equations where $\doteq$ appears. The known constants have been written explicitly. In particular, if one chooses the right constant so that \eqref{u-to-fi} is an equality, then all the other equations where $\doteq$ appears become equalities by using the same constant.
\end{rem}

\noindent
We would like to use the involution property of the Hankel transform to introduce the inverse relation of \eqref{u-to-fi}. Unfortunately we are outside the standard functional spaces where the involution property is valid, since for example an eigenfunction $\varphi$ of Proposition \ref{le-fi} satisfies $\varphi (t) = O(e^{\eps t})$ as $t\to \infty$ for all $\eps>0$. Hence we first explicitly construct the Hankel transform of the term $\chi_{q-1} \varphi$.

\begin{definition} \label{def-0-gmf-beta}
For any $q$ with $\Re(q)>0$ and $w\in \C \setminus (1,\infty)$, define the one-parameter family of functions
\begin{equation}\label{fi-to-u-beta}
u_\beta(iy) := \HHH_{q-\frac 12} [\exp_{-\beta}\, \chi_{q-1}\, \varphi](y)\, , \qquad \Re(\beta)>0
\end{equation}
for functions $\varphi:(0,+\infty) \to \C$ which make the integral converge.
\end{definition}

\noindent
Thanks to the properties of the Bessel function recalled in Section \ref{sec:sfit}, the integral in \eqref{fi-to-u-beta} is absolutely convergent if $\varphi$ is as in Proposition \ref{le-fi}, that is $\varphi$ is in $L^2(m_q)$, satisfies $(M+N_q)\varphi = \frac 1w \varphi$ and can be written as in \eqref{fi-power-b} with $a_{0,q}=b=0$ for $w=1$, and as in \eqref{fi-power-0} for $w\in \C \setminus [1,\infty)$. In fact by definition $\varphi$ satisfies $\varphi(t)=O(1)$ as $t\to 0^+$ and $\varphi (t) = O(e^{\eps t})$ as $t\to \infty$ for all $\eps>0$.

\begin{theorem}\label{ext-to-0}
For any $q$ with $\Re(q)>0$ and any $w\in \C \setminus (1,\infty)$,  and for $\varphi$ as in Proposition \ref{le-fi} with $b=0$, the function $u_\beta(iy)$ can be extended for all $y>0$ as an analytic function of $\beta$ to a small domain containing the origin. Moreover $u_0(iy)$ satisfies
\begin{equation}\label{la-u}
u_0(iy) = w \left[ g(y) + g\left( \frac 1y \right) \right] \,  ,\qquad \forall\, y>0
\end{equation}
where
$$g(y)= \HHH_{q-\frac 12}[\exp_{-1}\, \chi_{q-1}\, \varphi] (y) = \sum_{n=0}^{\infty} (-1)^{n}\, a_{n,q}\, \frac{y^{n+q}}{(1+y^{2})^{\frac n2 + \frac q2 + \frac 14}}\, \PPP^{-q+\frac 12}_{n+q-\frac 12} \left( \frac{y}{(1+y^{2})^{\frac 12}} \right)\, ,
$$
and $\{a_{n,q}\}$ is given in \eqref{fi-power-b} with $a_{0,q}=0$ for $w=1$, and in \eqref{fi-power-0} for $w\in \C \setminus [1,\infty)$.
\end{theorem}

\begin{proof}
Let us fix $y>0$. Using the functional equation $(M+N_q)\varphi = \frac 1w \varphi$, we can write
\begin{equation} \label{est-u-beta}
u_\beta(iy) = w\, \HHH_{q-\frac 12}[ \exp_{-\beta}\, \chi_{q-1} M\varphi](y) + w\, \HHH_{q-\frac 12}[ \exp_{-\beta}\, \chi_{q-1} N_q\varphi](y)
\end{equation}
since the first integral on the right hand side is absolutely convergent. Moreover we can change the order of integration in the second integral, that is
$$
\HHH_{q-\frac 12}[ \exp_{-\beta}\, \chi_{q-1} N_q\varphi](y) = \int_0^\infty\, J_{q-\frac 12}(ty) \sqrt{ty}\, e^{-\beta t} t^{q-1} \, \int_0^\infty \, J_{2q-1}(2\sqrt{st}) \left( \frac st \right)^{q-\frac 12}\, e^{-s} \varphi(s) \, ds \ dt =
$$
$$
= \int_0^\infty\, e^{-s} \, s^{q-1} \sqrt{s} \, \varphi(s) \, \int_0^\infty\, J_{q-\frac 12}(ty) \sqrt{ty}\, J_{2q-1}(2\sqrt{st})\, e^{-\beta t}\, t^{-\frac 12} dt \ ds
$$
since again the two-variable integral is absolutely convergent under the assumption $\Re(\beta)>0$. Hence, applying \cite[vol. II, eq. 8.12.(17), p. 58]{E2}, we get
$$
\HHH_{q-\frac 12}[ \exp_{-\beta}\, \chi_{q-1} N_q\varphi](y) = \int_0^\infty\, J_{q-\frac 12}\left( \frac{sy}{y^2+\beta^2} \right) \sqrt{sy}\ e^{-s-\frac{s\beta}{y^2+\beta^2}}\, \frac{s^{q-1}}{\sqrt{y^2+\beta^2}}\, \varphi(s)\ ds\, .
$$
The integral on the right hand side is absolutely convergent if
$$\Re\left( 1 + \frac{\beta}{y^2+\beta^2}\right) >0$$
hence the left hand side can be extended as an analytic function of $\beta$ to a small domain containing the origin. In particular we find that
\begin{equation} \label{est-2-to-0}
\HHH_{q-\frac 12}[ \exp_{-\beta}\, \chi_{q-1} N_q\varphi]\Big|_{\beta=0}(y) = \HHH_{q-\frac 12}[\exp_{-1}\, \chi_{q-1}\, \varphi] \left( \frac 1y \right)\, .
\end{equation}
Coming back to \eqref{est-u-beta}, the first term
$$
\HHH_{q-\frac 12}[ \exp_{-\beta}\, \chi_{q-1} M\varphi](y) =  \int_0^\infty\, J_{q-\frac 12}(ty) \sqrt{ty}\, e^{-\beta t} t^{q-1} \, e^{-t}\, \varphi(t) \, dt
$$
is absolutely convergent for $\Re(\beta)>-1$, hence again can be extended as an analytic function of $\beta$ to $\Re(\beta)>-1$, satisfying
\begin{equation} \label{est-1-to-0}
\HHH_{q-\frac 12}[ \exp_{-\beta}\, \chi_{q-1} M\varphi]\Big|_{\beta=0}(y) = \HHH_{q-\frac 12}[\exp_{-1}\, \chi_{q-1}\, \varphi] (y)\, .
\end{equation}
Hence, putting together \eqref{est-1-to-0} and \eqref{est-2-to-0}, we have proved that $u_\beta(iy)$ can be extended, as an analytic function of $\beta$, to a small domain containing the origin for all $y>0$, and
\begin{equation}\label{est-u-to-0}
u_0(iy) = w\, \HHH_{q-\frac 12}[\exp_{-1}\, \chi_{q-1}\, \varphi] (y) + w\, \HHH_{q-\frac 12}[\exp_{-1}\, \chi_{q-1}\, \varphi] \left( \frac 1y \right)\, .
\end{equation}
This establishes \eqref{la-u} with $g = \HHH_{q-\frac 12}[\exp_{-1}\, \chi_{q-1}\, \varphi]$.
We now use the power series expansion for $\varphi$ to obtain the series representations for $g$. 

\noindent
First we write $g(y) = G(y,\beta)\Big|_{\beta=0}$
where
$$
G(y,\beta) := \HHH_{q-\frac 12}[\exp_{-(1+\beta)}\, \chi_{q-1}\, \varphi] (y)
$$
for $y>0$ and $\Re(\beta)>-1$, the integral on the right hand side being absolutely convergent by the estimates used to justify the convergence in \eqref{fi-to-u-beta}. Then we use the identity
$$
\frac{w\, e^{-t}}{1-w\, e^{-t}} = \frac{1}{1-w\, e^{-t}} - 1
$$
in the definition of $G(y,\beta)$ to obtain
$$
G(y,\beta) = \int_{0}^{\infty}\, J_{q-\frac 12}(ty) \sqrt{y} \, e^{-(1+\beta)t}\, \frac{w\, e^{-t}}{1-w\, e^{-t}}\, \sum_{n=0}^{\infty} \frac{(-1)^{n}\, a_{n,q}\, t^{n+q-\frac 12}}{\Gamma(n+2q)}\, \ dt = 
$$
$$
= \frac 1w \int_{0}^{\infty}\, J_{q-\frac 12}(ty) \sqrt{y} \, \frac{w\, e^{-(1+\beta) t}}{1-w\, e^{-t}}\, \sum_{n=0}^{\infty} \frac{(-1)^{n}\, a_{n,q}\, t^{n+q-\frac 12}}{\Gamma(n+2q)}\ dt - \int_{0}^{\infty}\, J_{q-\frac 12}(ty) \sqrt{y} \, e^{-(1+\beta) t} \sum_{n=0}^{\infty} \frac{(-1)^{n}\, a_{n,q}\, t^{n+q-\frac 12}}{\Gamma(n+2q)}\ dt =
$$
$$
= \frac 1w\, u_{\beta}(iy) - \int_{0}^{\infty}\, J_{q-\frac 12}(ty) \sqrt{y} \, e^{-(1+\beta) t} \sum_{n=0}^{\infty} \frac{(-1)^{n}\, a_{n,q}\, t^{n+q-\frac 12}}{\Gamma(n+2q)}\ dt\, .
$$
Using \cite[vol II, p. 14]{E}, it holds
\[
|J_{q-\frac 12}(ty)| \le \frac{\sqrt{\pi}}{|\Gamma(q)|}\, \left( \frac 12 ty\right)^{\Re(q-\frac 12)}\, , \quad \forall\, t,y>0
\]
hence
$$
\sup_{t\in \R^{+}}\, \left| J_{q-\frac 12}(ty) \, e^{-(1+\beta)t}\,  t^{n+q-\frac 12} \right| \le \text{const.$(w,q,y)$}\, \sup_{t\in \R^{+}}\, \left| e^{-(1+\beta)t}\, t^{n+2q-1} \right| 
$$
where const.$(w,q,y)$ denotes a constant only depending on $w, q$ and $y$. Since
$$
\sup_{t\in \R^{+}}\, \left| e^{-(1+\beta)t}\, t^{n+2q-1} \right| \le e^{-n-2\Re(q)+1}\, \left| \frac{n+2q-1}{\Re(1+\beta)} \right|^{n+2\Re(q)-1}
$$
we find 
$$
\limsup_{n\to \infty}\, \left( \sup_{t\in \R^{+}}\, \left| \frac{a_{n,q}}{\Gamma(n+2q)}\, e^{-(1+\beta)t}\, t^{n+2q-1} \right| \right)^{\frac 1n} \le \frac{1}{\Re(1+\beta)} <1\, , \qquad \text{for $\Re(\beta)>0$}
$$
where we also used that $\limsup_{n} |a_{n,q}|^{\frac 1n} \le 1$. Hence we can write 
$$
G(y,\beta) = \frac 1w\, u_{\beta}(iy) - \sum_{n=0}^{\infty} \frac{(-1)^{n}\, a_{n,q}}{\Gamma(n+2q)}\ \HHH_{q-\frac 12} \left[ \exp_{-(1+\beta)}\, \chi_{n+q-1}\right](y) \, , \qquad \text{for $\Re(\beta)>0$}\, .
$$
Using now the proved analytic extension for $u_{\beta}$, we can write for the second term on the right hand side
\begin{equation} \label{ale} 
\sum_{n=0}^{\infty} \frac{(-1)^{n}\, a_{n,q}}{\Gamma(n+2q)}\ \HHH_{q-\frac 12} \left[ \exp_{-1}\, \chi_{n+q-1}\right](y) = \frac 1w\, u(iy) - G(y,0) = g\left( \frac 1y \right)\, .
\end{equation}
Moreover, using \cite[vol. II, eq. 8.6.(6), p. 29]{E2}, we obtain
$$
g\left( \frac 1y \right) = \sum_{n=0}^{\infty}\, (-1)^{n}\, a_{n,q}\, \frac{y^{\frac 12}}{(1+y^{2})^{\frac n2 + \frac q2 + \frac 14}}\, \PPP^{-q+\frac 12}_{n+q-\frac 12} \left( \frac{1}{(1+y^{2})^{\frac 12}} \right)
$$
and the proof is complete.
\end{proof}

\vskip 0.1cm
\noindent
We have thus proved the validity of the following expansion for $y\in (0,\infty)$
\begin{equation}\label{legendre-exp}
u_0(iy) = w\, \sum_{n=0}^{\infty}\, (-1)^{n}\, a_{n,q}\, \frac{y^{\frac 12}\, \PPP^{-q+\frac 12}_{n+q-\frac 12} \left( \frac{1}{(1+y^{2})^{\frac 12}} \right) + y^{n+q}\, \PPP^{-q+\frac 12}_{n+q-\frac 12} \left( \frac{y}{(1+y^{2})^{\frac 12}} \right)}{(1+y^{2})^{\frac n2 + \frac q2 + \frac 14}}\, .
\end{equation}
Moreover, letting $y=\tan \vartheta$ with $\vartheta \in (0,\frac \pi 2)$ in \eqref{legendre-exp}, we get 
\begin{equation}\label{legendre-exp-theta}
u_0(iy) = w\, (\sin \vartheta \, \cos \vartheta)^{\frac 12}\, \sum_{n=0}^{\infty}\, (-1)^{n}\, a_{n,q}\, \Big[ (\cos \vartheta)^{n+q-\frac 12}\, \PPP^{-q+\frac 12}_{n+q-\frac 12} \left( \cos \vartheta \right) + (\sin \vartheta)^{n+q-\frac 12}\, \PPP^{-q+\frac 12}_{n+q-\frac 12} \left( \sin \vartheta \right) \Big]\, ,
\end{equation}
and from the integral representation valid for $\xi>0$ (see \cite[vol. I, eq. (27), p. 159]{E})
$$
\PPP^{-q+\frac 12}_{n+q-\frac 12} \left( \cos \vartheta \right) = \frac{\sqrt{2}\, (\sin \vartheta)^{-q+\frac 12}}{\sqrt{\pi}\, \Gamma(q)}\, \int_{0}^{\vartheta}\, \frac{\cos((n+q)t)}{(\cos t - \cos \vartheta)^{1-q}}\ dt
$$
and the similar one for $\PPP^{-q+\frac 12}_{n+q-\frac 12} \left( \sin \vartheta \right)$, we see that the convergence in \eqref{legendre-exp-theta} is uniform on any compact interval contained in $(0,\frac \pi 2)$. Hence the convergence in \eqref{legendre-exp} is uniform on any compact interval contained in $(0,\infty)$.

\begin{corollary}\label{cor-maass-even}
Letting $w=1$ and $\{a_{n,q}\}$ as in \eqref{fi-power-b} with $a_{0,q}=0$, the function $u_0(iy)$ in \eqref{legendre-exp} is the restriction to the imaginary axis of an even Maass cusp form.
\end{corollary}
 \begin{proof}
 It follows from the fundamental theorem of Maass (see \cite[Theorem 2, p. 234]{Ter} and \cite[Proposition 2.1]{Le}) that even Maass cusp forms are uniquely determined as functions with restriction on the imaginary axis of the form \eqref{four-exp} for $x=0$ and coefficients $\{c_{n,q}\}$ which make the series \eqref{four-exp} satisfy $u(iy)= u(i\frac 1y)$.

\noindent 
By definition we have that the function $u_0(iy)$ in \eqref{la-u} corresponds to an eigenfunction $\varphi$ of the operator $M+N_q$ as explained in Proposition \ref{le-fi}, and by \eqref{u-to-fi} and the involution property of the Hankel transform, it admits a Fourier expansion as in \eqref{four-exp}. Moreover by the properties of $u_0(iy)$ found in Theorem \ref{ext-to-0} it also satisfies $u_0(iy)= u_0(i\frac 1y)$. Hence the proof is finished.
\end{proof}

\begin{rem}\label{rem-ref}
A consequence of this result is that Maass Theorem can be reformulated by saying that even Maass cusp forms are uniquely determined as functions with restriction on the imaginary axis of the form \eqref{legendre-exp} and coefficients $\{a_{n,q}\}$ which satify $\limsup_{n} |a_{n,q}|^{\frac 1n} \le 1$ and the identity
\begin{equation} \label{iden}
\sum_{n=1}^\infty\, a_{n,q}\, z^n = \sum_{n=1}^\infty\, a_{n,q}\, \Big( (z-1)^n - \frac{(-1)^n}{(z+1)^{2q+n}}\Big)
\end{equation}
for all $z$ where the series converge. This identity follows from Proposition \ref{le-fi} and \cite{BI}.

\noindent
As far as numerical computations are concerned, we also remark that \eqref{iden} has been reduced in \cite{imen,BGI} to a linear algebra identity for infinite matrices in the case $q$ real. The same can be done for general complex values of $q$ with positive real part (unpublished notes). 
\end{rem}

\subsection{The even case for $b$-gpf} \label{even-b-good}

We now extend Theorem \ref{ext-to-0} to the case of even $b$-gpf, which do exist only for $w=1$. We recall that non-cuspidal Maass forms of eigenvalue $\lambda$ form a one-dimensional subspace which is spanned by the non-holomorphic Eisenstein series defined for $\xi>1$ as
\begin{equation}\label{eisen-q}
E(z,q) = \zeta(2q)\, y^{q} \left(1+ \frac{1}{|z|^{2q}} \right) + 2\, \sum_{c,d\ge 1} \left( \frac{y}{|cz+d|^{2}} \right)^q\, , \qquad z=x+i y\, , 
\end{equation}
and extended to $\C$ as a meromorphic function with a simple pole at $q=1$ with residue the constant function $\frac \pi 2$, by the Fourier series expansions
\begin{equation}\label{eisen}
E(x+i y,q) = \zeta(2q)\, y^{q} + \frac{\pi^{\frac 12}\, \Gamma(q-\frac 12)}{\Gamma(q)}\, \zeta(2q-1)\, y^{1-q} + y^{\frac 12}\, \sum_{n\ge 1}\, \tilde c_{n,q}\, K_{q-\frac 12}(2\pi ny)\, \cos(2\pi nx)
\end{equation}
where
$$
\tilde c_{n,q} = \frac{4\pi^{q}}{\Gamma(q)}\, n^{\frac 12 -q}\, \sum_{d|n}\, d^{2q-1}\, .
$$
Notice that the extension of $E(x+iy,q)$ to $\C$ has no pole at $q=\frac 12$ since the contribution from the term $\zeta(2q)$ is cancelled by the contribution of the term containing $\Gamma(q-\frac 12)$. Moreover it is proved in \cite{CM1} and \cite{LeZa} (see the proof of equation (2.30) and page 243) that the function $\psi_{q}^{+}$ defined in \eqref{zagier-psi} for $\xi>1$, which is an eigenfunction of $\PP^+_q$ with eigenvalue $\lambda=1$, satisfies
\begin{equation}\label{eis-to-psi}
\psi_{q}^{+}(z) = \frac{\zeta(2q)}{2} \left( 1 + z^{-2q} \right) + \frac{2^{-q-\frac 12}}{\Gamma(q+\frac 12)}\, \LLL \left[ \chi_{q}\, \HHH_{q-\frac 12}[\tilde E(iy,q)]\right](z)
\end{equation}
where
\begin{equation}\label{eisen-q-tilde}
\tilde E(iy,q) = 2\, \sum_{c,d\ge 1} \left( \frac{y}{c^2y^2+d^{2}} \right)^q\, .
\end{equation}
It is shown in \cite{LeZa} that the function $\frac{\Gamma(2q)}{\Gamma(q-1)}\, \psi_q^+$ can be analytically continued to $\C$, we give here a proof of this fact for $\{\xi>0\}$ using the $\BBB_q$ transform. 

\begin{theorem}\label{zagier-funct-ser}
The equation
\begin{equation}\label{for-ps-zf}
\psi_{q}^{+}(z) = \BBB_{q}\left[\frac{\zeta(2q)}{2}\, \frac{\delta_{0}(t)}{t^{2q-1}} + \frac{e^{-t}}{1-e^{-t}} \, \sum_{n=0}^{\infty}\, \frac{(-1)^{n} a_{n,q}\, t^{n}}{\Gamma(n+2q)}   \right](z)
\end{equation}
where\footnote{Here by $\delta_0$ we denote the Dirac delta function at 0, and we use its definition when it is used as argument of an integral transform.}
$$
\left\{
\begin{array}{l}
a_{0,q} = \zeta(2q-1)\\[0.2cm]
a_{n,q} = (-1)^{n}\, \frac{\Gamma(n+2q)}{n!\, \Gamma(2q)}\, \left( \frac{\zeta(2q)}{2} + \frac{1}{n+1}\, \sum_{i=0}^{n}\, {n+1\choose i}\, B_{i}\, \zeta(2q-1+i) \right)\, , \qquad n\ge 1
\end{array}
\right.
$$
defines a meromorphic extension of $\psi_{q}^{+}(z)$ to $\{\xi>0\}$ with simple pole at $q=1$ and residue the function $\frac{1}{2z}$, which is the density of the invariant measure of the Farey map, up to a multiplicative constant.
\end{theorem}

\begin{proof}
We first show that expression \eqref{for-ps-zf} coincides with the definition \eqref{zagier-psi} of the function $\psi_q^+$ for $\xi >1$. Then we show the meromorphic extension of \eqref{for-ps-zf} to the half-plane $\{\xi >0\}$.

\noindent
We first use \cite[Remark 2.6]{BI} and in particular
\begin{equation}\label{due}
\BBB_{q}\left[ \frac{\zeta(2q)}{2}\, \frac{\delta_{0}(t)}{t^{2q-1}} \right](z) = \frac{\zeta(2q)}{2} z^{-2q}
\end{equation}
to obtain the second term on the right hand side of \eqref{zagier-psi}. The first term is obtained by 
\begin{equation}\label{uno}
\frac{\zeta(2q)}{2} = \BBB_{q}\left[ \frac{\zeta(2q)}{2\, \Gamma(2q)}\right](z) = \BBB_{q}\left[ \frac{\zeta(2q)}{2\, \Gamma(2q)}\, \frac{e^{-t}}{1-e^{-t}} \, \sum_{n=1}^{\infty}\, \frac{t^{n}}{n!} \right] (z)\, .
\end{equation}
For the other terms we argue as follows
$$
\sum_{m,n\ge 1} \, \frac{1}{(mz+n)^{2q}} = \frac{1}{z^{2q}}\,
\sum_{m,n\ge 1} \, \frac{1}{n^{2q} \left( \frac m n + \frac 1 z
\right)^{2q}} = \frac{1}{\Gamma(2q)\, z^{2q}}\, \sum_{m,n\ge 1} \,
\frac{1}{n^{2q}}\, \LLL \left[ t^{2q-1} e^{-\frac m n\, t}
\right] \left(\frac 1 z \right) =
$$
$$
= \frac{1}{\Gamma(2q)}\, \sum_{m,n\ge 1} \, \frac{1}{n^{2q}}\,
\BBB_q \left[ e^{-\frac m n\, t} \right](z) = \BBB_q \left[
\frac{1}{\Gamma(2q)}\, \sum_{m,n\ge 1} \, \frac{e^{-\frac m n\,
t}}{n^{2q}}\, \right](z) = \BBB_q \left[ \frac{1}{\Gamma(2q)}\,
\sum_{n\ge 1} \, \frac{1}{n^{2q}}\, \frac{e^{-\frac t
n}}{1-e^{-\frac t n}} \right](z)\, .
$$
Since $\xi>1$, we can write
$$
\sum_{n\ge 1} \, \frac{1}{n^{2q}}\, \frac{e^{-\frac t
n}}{1-e^{-\frac t n}} = \frac{e^{-t}}{1-e^{-t}}\, \sum_{n\ge 1} \, \frac{1}{n^{2q}}\, \frac{e^{t}-1}{e^{\frac tn}-1} = \frac{e^{-t}}{1-e^{-t}}\, \sum_{n\ge 1} \, \frac{1}{n^{2q}}\, \sum_{j=0}^{n-1}\, (e^{\frac tn})^{j} =
$$
$$
= \frac{e^{-t}}{1-e^{-t}}\, \left[ \sum_{n\ge 1} \, \frac{1}{n^{2q}}\, + \sum_{k\ge 0} \left( \sum_{n\ge 2} \, \sum_{j=1}^{n-1}\, \frac{j^{k}}{n^{2q+k}} \right) \frac{t^{k}}{k!} \right]= \frac{e^{-t}}{1-e^{-t}}\, \sum_{k\ge 0} A_{k,q}\, \frac{t^{k}}{k!}
$$
with
$$
A_{0,q} = \zeta(2q) + \sum_{n\ge 2}\, \frac{n-1}{n^{2q}} = \zeta(2q-1)
$$
and in general 
$$
A_{k,q} = \sum_{n\ge 2}\, \frac{S_{k}(n-1)}{n^{2q+k}}\, , \quad k\ge 1
$$
where $S_{k}(n-1) = \sum_{j=1}^{n-1}\, j^{k}$. Notice that $S_{k}(n-1) \le n^{k+1}$, thus for $\xi>1$ the sum defining $A_{k,q}$ is convergent and $|A_{k,q}|\le \zeta(2\xi-1)$ for all $k\ge 1$. Hence the series $\sum_{k\ge 0} A_{k,q}\, \frac{t^{k}}{k!}$ converges for $t\in \R$ and
\begin{equation}\label{tre}
\sum_{m,n\ge 1} \, \frac{1}{(mz+n)^{2q}} = \BBB_{q} \left[ \frac{1}{\Gamma(2q)}\, \frac{e^{-t}}{1-e^{-t}}\, \sum_{k\ge 0} A_{k,q}\, \frac{t^{k}}{k!} \right](z)\, .
\end{equation}
Moreover, we recall that
$$
S_{k}(n) =  \frac{1}{k+1}\, n^{k+1} + \frac 12 \, n^k +  \frac{1}{k+1}\, \sum_{i=2}^{k}\, {k+1 \choose i}\, B_{i}\, n^{k+1-i}
$$
where $B_{i}$ are the Bernoulli numbers. Hence for $\xi>1$ and $k\ge 1$
$$
A_{k,q} = \sum_{n\ge 2}\, \frac{S_{k}(n)-n^k}{n^{2q+k}} = \sum_{n\ge 2}\, \frac{S_{k}(n) - \frac{1}{k+1}\, n^{k+1} - \frac 12 \, n^k}{n^{2q+k}} + \frac{1}{k+1}\, \Big(\zeta(2q-1)-1\Big) - \frac 12\, \Big(\zeta(2q)-1\Big) =
$$
$$
= \frac{1}{k+1}\, \Big(\zeta(2q-1)-1\Big) - \frac 12\, \Big(\zeta(2q)-1\Big) + \frac{1}{k+1}\, \sum_{i=2}^{k}\, {k+1\choose i}\, B_{i}\, \Big(\zeta(2q-1+i)-1\Big) =
$$
\[
= \frac{1}{k+1}\, \sum_{i=0}^{k}\, {k+1\choose i}\, B_{i}\, \Big(\zeta(2q-1+i)-1\Big) = \frac{1}{k+1}\, \sum_{i=0}^{k}\, {k+1\choose i}\, B_{i}\, \zeta(2q-1+i)\, ,
\]
using the identity $\sum_{i=0}^{n}\, {n+1\choose i}\, B_{i} =0$. These expressions for $A_{k,q}$ are holomorphic in $\set{\xi>0}$ for all $k\ge 0$ except for simple poles at $q=\frac 12$ and $q=1$. Moreover, using
\begin{equation}\label{vero}
\left| S_{k}(n) - \frac{1}{k+1}\, n^{k+1} - \frac 12 \, n^k \right| \le const\, k^{2}\, n^{k-1}\, , \qquad k\ge 1
\end{equation}
which is proved in the Appendix A, we have that
$$
\left| \sum_{n\ge 2}\, \frac{S_{k}(n) - \frac{1}{k+1}\, n^{k+1} - \frac 12 \, n^k}{n^{2q+k}} \right| \le const\, k^{2}\, \sum_{n\ge 2}\, \frac{n^{k-1}}{n^{2\xi+k}} = const\, k^{2}\, \zeta(2\xi+1)
$$
for all $q$ in $\set{\xi>0}$, hence $|A_{k,q}| = O(k^{2})$ for all $q$ in $\set{\xi>0}$. This implies that \eqref{tre} is valid for $\xi>0$, and putting together \eqref{due}, \eqref{uno} and \eqref{tre}, we get for $\xi>0$
$$
\psi_{q}^{+}(z) = \BBB_{q}\left[ \frac{\zeta(2q)}{2}\, \frac{\delta_{0}(t)}{t^{2q-1}} + \frac{e^{-t}}{1-e^{-t}} \, \sum_{n=0}^{\infty}\, \frac{(-1)^{n} a_{n,q}\, t^{n}}{\Gamma(n+2q)}\right](z)
$$
where
$$
\left\{
\begin{array}{l}
a_{0,q} = \zeta(2q-1)\\[0.2cm]
a_{n,q} = (-1)^{n}\, \frac{\Gamma(n+2q)}{n!\, \Gamma(2q)}\, \left( \frac{\zeta(2q)}{2} + A_{n,q} \right)\, , \qquad n\ge 1
\end{array}
\right.
$$
which are holomorphic except for a simple pole at $q=1$.

\noindent
There is also a pole at $q=\frac 12$ in the coefficient $\frac{\zeta(2q)}{2}$ of the first term in the argument of the $\BBB_q$ transform. However, when applying the $\BBB_q$, we obtain that $\psi_{q}^{+}$ can be written as in \eqref{forma-eig-p} with $c=\frac{\zeta(2q)}{2}$ and $b=\zeta(2q-1)$, so the first two terms are given by
\[
\frac{\zeta(2q)}{2}\, \frac{1}{z^{2q}} + \frac{\zeta(2q-1)\, \Gamma(2q-1)}{\Gamma(2q)}\, \frac 1z
\]
so that there is no pole at $q=\frac 12$, as it happens for the Eisenstein series in \eqref{eisen}.

\noindent
Finally we can compute the residue for $\psi_{q}^{+}$ at $q=1$ using \eqref{for-ps-zf}. The only contributing terms are those containing $\zeta(2q-1)$, which has residue $\frac 12$. Hence $\text{Res}_{q=1}(a_{n,q}) = \frac{(-1)^{n}}{2}$ and
$$
\text{Res}_{q=1}\left[ \frac{\zeta(2q)}{2}\, \frac{\delta_{0}(t)}{t^{2q-1}} + \frac{e^{-t}}{1-e^{-t}} \, \sum_{n=0}^{\infty}\, \frac{(-1)^{n} a_{n,q}\, t^{n}}{\Gamma(n+2q)}  \right] = \frac{e^{-t}}{1-e^{-t}} \, \sum_{n=0}^{\infty}\, \frac{t^{n}}{2\, \Gamma(n+2)}
$$
which gives
$$
\text{Res}_{q=1}(\psi_{q}^{+}) = \BBB_{1}\left[ \frac{e^{-t}}{1-e^{-t}} \, \sum_{n=0}^{\infty}\, \frac{t^{n}}{2\, \Gamma(n+2)}\right](z) =  \BBB_{1}\left[\frac {1}{2t}\right](z) = \frac{1}{2z}\, .
$$
This concludes the proof.
\end{proof}

\vskip 0.5cm
\noindent
By Theorem C-(a), the function $\psi_q^+$ satisfies the equation $\II_q \psi_q^+ = \psi_q^+$, as is easily verified using the definition \eqref{zagier-psi}. Then, using \eqref{tre} and \eqref{eis-to-psi} it follows that the function
\begin{equation}\label{phi-tilde}
\tilde \varphi (t) := \frac{1}{\Gamma(2q)}\, \frac{e^{-t}}{1-e^{-t}}\, \sum_{k\ge 0} A_{k,q}\, \frac{t^{k}}{k!}
\end{equation}
satisfies 
\[
\BBB_q[\tilde \varphi] = \LLL[\chi_{2q-1}\, \tilde \varphi] = \frac{2^{-q-\frac 12}}{\Gamma(q+\frac 12)}\, \LLL \left[ \chi_{q}\, \HHH_{q-\frac 12}[\tilde E(iy,q)]\right]\, ,
\]
from which we get the analogue of \eqref{u-to-fi}
\begin{equation}\label{eis-to-fi}
\tilde \varphi (t) = \frac{2^{-q-\frac 12}}{\Gamma(q+\frac 12)}\, t^{1-q}\, \HHH_{q-\frac 12}[\tilde E(iy,q)](t)\, ,
\end{equation}
for $\tilde E(iy,q)$ defined in \eqref{eisen-q-tilde}. From this we get an analytic continuation of $E(iy,q)$ different from the Fourier series expansion \eqref{eisen}.

\begin{theorem}\label{cor-eis}
The function $U(iy)$ defined by
\begin{equation}\label{fi-for-psi}
\begin{aligned}
U(iy):= \, & \zeta(2q)\, \Big( y^{q} + y^{-q} \Big) - 2\, \zeta(2q)\, \left( \frac{y}{1+y^2} \right)^q +\\[0.2cm]
& + 2^{q+\frac 12}\, \Gamma\left(q+\frac 12\right) \sum_{n=0}^{\infty}\, (-1)^{n}\, b_{n,q}\, \frac{y^{\frac 12}\, \PPP^{-q+\frac 12}_{n+q-\frac 12} \left( \frac{1}{(1+y^{2})^{\frac 12}} \right) + y^{n+q}\, \PPP^{-q+\frac 12}_{n+q-\frac 12} \left( \frac{y}{(1+y^{2})^{\frac 12}} \right)}{(1+y^{2})^{\frac n2 + \frac q2 + \frac 14}} 
\end{aligned}
\end{equation}
with
\[
b_{n,q} = (-1)^{n}\, \frac{\Gamma(n+2q)}{(n+1)!\, \Gamma(2q)}\, \sum_{i=0}^{n}\, {n+1\choose i}\, B_{i}\, \zeta(2q-1+i)\, ,
\]
gives an analytic continuation of the Eisenstein series $E(iy,q)$ in \eqref{eisen-q} to $q\in \C$ with a simple pole at $q=1$ with residue the constant function $\frac \pi 2$.
\end{theorem}

\begin{proof}
Writing the Eisenstein series $E(iy,q)$ as in \eqref{eisen-q}
\[
E(iy,q) = \zeta(2q)\, \Big( y^{q} + y^{-q} \Big) + \tilde E(iy,q)\, ,
\]
we proceed as in Theorem \ref{ext-to-0} to invert the relation \eqref{eis-to-fi}.

\noindent
The proof follows the same lines as that of Theorem \ref{ext-to-0} with some modifications. The first is that the function $\tilde \varphi$ satisfies the functional equation
\begin{equation}\label{eq-funz-tilde}
((M+N_q) \tilde \varphi)(t) = \tilde \varphi(t) - \frac{\zeta(2q)}{\Gamma(2q)}\, e^{-t}\, .
\end{equation}
This follows by applying $\PP^+_q$ to $\psi_q^+(z)= \frac{\zeta(2q)}{2} \, (1+z^{-2q}) +\BBB_q[\tilde \varphi] (z)$. Indeed $\tilde \varphi$ is of the right form to apply \eqref{trasf-op-l2}, hence
\[
\psi_q^+(z) = (\PP_q^+ \psi_q^+) (z) = \PP_q^+ \left( \frac{\zeta(2q)}{2} \, (1+z^{-2q})\right) + (\PP_q^+ \BBB_q[\tilde \varphi] )(z) =
\]
\[
=  \frac{\zeta(2q)}{2} \, (1+z^{-2q}) + \zeta(2q)\, (1+z)^{-2q}\, + \BBB_q[(M+N_q)\tilde \varphi)](z)\, .
\]
Using
\[
(1+z)^{-2q} = \frac{1}{\Gamma(2q)}\, z^{-2q}\, \int_0^\infty\, e^{-t(1+\frac 1z)}\, t^{2q-1}\, dt = \frac{1}{\Gamma(2q)}\, \BBB_q[e^{-t}](z)
\]
we obtain \eqref{eq-funz-tilde}.

\noindent
Letting now
\[
\tilde U_\beta(iy) := 2^{q+\frac 12}\, \Gamma\left(q+\frac 12\right)\, \HHH_{q-\frac 12} [\exp_{-\beta}\, \chi_{q-1}\, \tilde \varphi](y)\, , \qquad \Re(\beta)>0
\]
we get from \eqref{eq-funz-tilde}
\[
\tilde U_\beta(iy) = 2^{q+\frac 12}\, \Gamma\left(q+\frac 12\right)\, \Big( \HHH_{q-\frac 12} [\exp_{-\beta}\, \chi_{q-1}\, (M+N_q)\tilde \varphi](y) + \frac{\zeta(2q)}{\Gamma(2q)} \HHH_{q-\frac 12} [\exp_{-\beta-1}\, \chi_{q-1}](y) \Big)\, .
\]
For the first term on the right hand side, for $\xi > \frac 12$ we can repeat the arguments of the proof of Theorem \ref{ext-to-0} leading to \eqref{est-u-to-0}, to get
\[
\HHH_{q-\frac 12} [\exp_{-\beta}\, \chi_{q-1}\, (M+N_q)\tilde \varphi]\big|_{\beta=0}(y) = \HHH_{q-\frac 12} [\exp_{-1}\, \chi_{q-1}\, \tilde \varphi](y) + \HHH_{q-\frac 12} [\exp_{-1}\, \chi_{q-1}\,  \tilde \varphi]\left( \frac 1y \right)\, ,
\] 
whereas the second term is absolutely convergent for $\beta=0$, thus we simply have
\[
\HHH_{q-\frac 12} [\exp_{-\beta-1}\, \chi_{q-1}]\big|_{\beta=0}(y) = \HHH_{q-\frac 12} [\exp_{-1}\, \chi_{q-1}](y)\, .
\]
Hence we obtain the continuation of $\tilde U_\beta$ to a neighborhood of $\beta=0$, and define $\tilde U(iy) := \tilde U_0(iy)$ by
\begin{equation}\label{est-U-to-0}
\tilde U(iy) = 2^{q+\frac 12}\, \Gamma\left(q+\frac 12\right)\, \Big( \HHH_{q-\frac 12} [\exp_{-1}\, \chi_{q-1}\, \tilde \varphi](y) + \HHH_{q-\frac 12} [\exp_{-1}\, \chi_{q-1}\,  \tilde \varphi]\left( \frac 1y \right) + \frac{\zeta(2q)}{\Gamma(2q)} \HHH_{q-\frac 12} [\exp_{-1}\, \chi_{q-1}](y) \Big)\, .
\end{equation}
To finish the proof, we use \eqref{phi-tilde} to write $\tilde \varphi$ as 
\[
\tilde \varphi(t) = \frac{e^{-t}}{1-e^{-t}} \, \sum_{n=0}^{\infty}\, \frac{(-1)^{n} b_{n,q}\, t^{n}}{\Gamma(n+2q)} 
\]
with $b_{n,q} = (-1)^n\, \frac{\Gamma(n+2q)}{n!\, \Gamma(2q)}\, A_{n,q}$, hence
\[
\left\{
\begin{array}{l}
b_{0,q} = \zeta(2q-1)\\[0.2cm]
b_{n,q} = (-1)^{n}\, \frac{\Gamma(n+2q)}{(n+1)!\, \Gamma(2q)}\, \sum_{i=0}^{n}\, {n+1 \choose i}\, B_{i}\, \zeta(2q-1+i)\, , \quad n\ge 1
\end{array}
\right.
\]
\noindent
Then we define as in the proof of Theorem \ref{ext-to-0} the function $\tilde g(y) = \tilde G(y,\beta)\Big|_{\beta=0}$
with
$$
\tilde G(y,\beta) := \HHH_{q-\frac 12}[\exp_{-(1+\beta)}\, \chi_{q-1}\, \tilde \varphi] (y)
$$
and repeat the same argument used in the proof of Theorem \ref{ext-to-0} to prove \eqref{ale}, to show that
\[
\sum_{n=0}^{\infty} \frac{(-1)^{n}\, b_{n,q}}{\Gamma(n+2q)}\ \HHH_{q-\frac 12} \left[ \exp_{-1}\, \chi_{n+q-1}\right](y) = \frac{2^{-q-\frac 12}}{\Gamma(q+\frac 12)}\, \tilde U(iy) - \tilde G(y,0) = \tilde g\left( \frac 1y\right) + R(y)\, ,
\]
with $R(y) := \frac{\zeta(2q)}{\Gamma(2q)} \HHH_{q-\frac 12} [\exp_{-1}\, \chi_{q-1}](y)$. The above equation can be used to obtain an expression for $\tilde g\left( \frac 1y\right)$ and the analogous for $\tilde g(y)$, that when substituted in \eqref{est-U-to-0} finally give
\[
\frac{2^{-q-\frac 12}}{\Gamma(q+\frac 12)}\, \tilde U(iy) =  \sum_{n=0}^{\infty} \frac{(-1)^{n}\, b_{n,q}}{\Gamma(n+2q)}\ \HHH_{q-\frac 12} \left[ \exp_{-1}\, \chi_{n+q-1}\right](y) + \sum_{n=0}^{\infty} \frac{(-1)^{n}\, b_{n,q}}{\Gamma(n+2q)}\ \HHH_{q-\frac 12} \left[ \exp_{-1}\, \chi_{n+q-1}\right]\left( \frac 1y\right) - R\left( \frac 1y\right)\, .
\]
The last step of the proof consists of the calculations of the Hankel transforms. The first one is the same as in Theorem \ref{ext-to-0}, that is
\[
 \sum_{n=0}^{\infty} \frac{(-1)^{n}\, b_{n,q}}{\Gamma(n+2q)}\ \HHH_{q-\frac 12} \left[ \exp_{-1}\, \chi_{n+q-1}\right](y) = \sum_{n=0}^{\infty}\, (-1)^{n}\, b_{n,q}\, \frac{y^{\frac 12}}{(1+y^{2})^{\frac n2 + \frac q2 + \frac 14}}\, \PPP^{-q+\frac 12}_{n+q-\frac 12} \left( \frac{1}{(1+y^{2})^{\frac 12}} \right)\, ,
 \]
 and the second one is
 \[
 R(y) = \frac{\zeta(2q)}{\Gamma(2q)}\, \frac{2^{q-\frac 12}\, \Gamma(q)}{\pi^{\frac 12}}\, \left( \frac{y}{1+y^2} \right)^q = \zeta(2q)\, \frac{2^{-q+\frac 12}}{\Gamma(q+\frac 12)}\, \, \left( \frac{y}{1+y^2} \right)^q 
\]
where we have used \cite[vol. II, eq. 8.6.(5), p. 29]{E2} and $\Gamma(2q) = \pi^{-\frac 12}\, 2^{2q-1}\, \Gamma(q)\, \Gamma(q+\frac 12)$.

\noindent
In the proof of Theorem \ref{zagier-funct-ser} we proved that $|A_{n,q}| = O(n^2)$, hence arguing as in \eqref{legendre-exp-theta}, we obtain that the expansion \eqref{fi-for-psi} is well defined for all $q\in \C$, except $q=\frac 12$ and $q=1$, and for all $y>0$. Moreover it is uniformly convergent in $y$ on any compact interval contained in $(0,\infty)$.

\noindent
We now first show that the expression \eqref{fi-for-psi} has no pole at $q=\frac 12$. It is enough to show that the term multiplying $\zeta(2q)$ vanishes at $q=\frac 12$, indeed
\[
\lim_{q\to \frac 12}\, (2q-1)\, U(iy) = y^{\frac 12} + y^{-\frac 12} - 2 \left( \frac{y}{1+y^2} \right)^{\frac 12} - \sum_{n=1}^{\infty}\, \frac{y^{\frac 12}\, P_{n} \left( \frac{1}{(1+y^{2})^{\frac 12}} \right) + y^{n+\frac 12}\, P_n \left( \frac{y}{(1+y^{2})^{\frac 12}} \right)}{(1+y^{2})^{\frac n2 + \frac 12}}
\]
where $P_n$ are the Legendre polynomials. Using equation (see \cite[eq. 18.12.11, p. 449]{nist})
\[
\sum_{n=0}^\infty\, P_n(\alpha)\, \beta^n = (1-2\alpha \beta +\beta^2)^{-\frac 12}
\]
for $\alpha \in (0,1)$ and $|\beta|< 1$, we obtain
\[
\sum_{n=1}^{\infty}\, \frac{y^{\frac 12}\, P_{n} \left( \frac{1}{(1+y^{2})^{\frac 12}} \right) + y^{n+\frac 12}\, P_n \left( \frac{y}{(1+y^{2})^{\frac 12}} \right)}{(1+y^{2})^{\frac n2 + \frac 12}} = \left( \frac{y}{1+y^2} \right)^{\frac 12} \left( \frac{(1+y^2)^{\frac 12}}{y} + (1+y^2)^{\frac 12} -2\right)\, ,
\]
hence
\[
\lim_{q\to \frac 12}\, (2q-1)\, U(iy) = 0\, .
\]
At $q=1$, the expression \eqref{fi-for-psi} has instead a pole with a residue that can be computed using $\text{Res}_{q=1}(b_{n,q}) = \frac{(-1)^{n}}{2}$. Letting $y=\tan \vartheta$ as above we find
$$
\text{Res}_{q=1}(U)(i y) = 2^{\frac 12}\, \Gamma\left( \frac 32 \right)\, (\sin \vartheta\, \cos \vartheta )^{\frac 12}\, \sum_{n=0}^{\infty}\, \, \Big[ (\cos \vartheta)^{n+\frac 12}\, \PPP^{-\frac 12}_{n+\frac 12} \left( \cos \vartheta \right) + (\sin \vartheta)^{n+\frac 12}\, \PPP^{-\frac 12}_{n+\frac 12} \left( \sin \vartheta \right) \Big]
$$
Using \cite[eq. 14.5.12, p. 359]{nist} we get
$$
(\sin \vartheta)^{\frac 12}\, (\cos \vartheta)^{n+1} \, \PPP^{-\frac 12}_{n+\frac 12} \left( \cos \vartheta \right) = \frac{1}{n+1}\, \left( \frac 2 \pi\right)^{\frac 12}\, (\cos \vartheta)^{n+1} \, \sin \left( (n+1) \vartheta \right)
$$
Hence
$$
\sum_{n=0}^{\infty}\, (\sin \vartheta)^{\frac 12}\, (\cos \vartheta)^{n+1} \, \PPP^{-\frac 12}_{n+\frac 12} \left( \cos \vartheta \right)  = \left( \frac 2 \pi\right)^{\frac 12}\, \Im \left( \sum_{n=0}^{\infty}\, \frac{1}{n+1}\, \left( \frac{1+\exp(2i\vartheta)}{2} \right)^{n+1} \right) =
$$
$$
= \left( \frac 2 \pi\right)^{\frac 12}\, \Im \left( -\log\left( \frac{1-\exp(2i\vartheta)}{2} \right)  \right) = - \left( \frac 2 \pi\right)^{\frac 12}\, \arctan \left(-\frac{\sin (2\vartheta)}{1-\cos(2\vartheta)}\right) = \left( \frac 2 \pi\right)^{\frac 12}\, \arctan \frac 1y
$$
recalling that $y=\tan \vartheta$.
Hence finally
$$
\text{Res}_{q=1}(U)(iy) = 2^{\frac 12}\, \Gamma\left( \frac 32 \right)\, \left( \frac 2 \pi\right)^{\frac 12}\, \Big( \arctan y + \arctan \frac 1y\Big) = \frac \pi 2
$$
recalling $\Gamma(\frac 32) = \frac 1 2\, \pi^{\frac 12}$.
\end{proof}

\subsection{The odd case for 0-gpf} \label{sec:odd-good}
Here we are going to repeat the approach of Section \ref{sec:even-good} for odd period functions, which in the case $w=1$ are in one-to-one correspondence with the set of odd Maass cusp forms, as shown in \cite{LeZa}. Also in this case it is fundamental to use the Fourier series expansion for the odd cusp forms given by
\begin{equation}\label{four-exp-odd}
u(x+iy) = y^{\frac 12}\, \sum_{n\ge 1}\, c_{n,q}\, K_{q-\frac 12}(2\pi n y)\, \sin(2\pi nx)\, ,
\end{equation}
with $c_{n,q}$ having at most a polynomial growth. The integral correspondence between even cusp forms and even period functions proved in \cite{Le} is extended to the odd case in \cite[Section II.1]{LeZa}. We can formally proceed as for \eqref{formula-1} by applying \cite[Proposition 4.3]{Le} to get\footnote{Proposition 4.3 in \cite{Le} can be applied only for $\xi > \frac 32$.}
\[
\psi(z) = - \frac 1z\, \LLL\left[ \chi_{q-1} \HHH_{q-\frac 32} [y\, u_x(i y)] \right] (z)\, ,
\]
where $u_x = \frac{\partial}{\partial x} u$, and find by \eqref{invo-fi} that 
\[
\frac 1z\, \LLL\left[ \chi_{q-1} \HHH_{q-\frac 32} [y\, u_x(i y)] \right] (z) \doteq \LLL\left[ \chi_{2q-1} \varphi \right] (z)
\]
for a function $\varphi \in L^2(m_q)$ satisfying $(M-N_q)\varphi = \varphi$, with expansion as in \eqref{fi-power-b} with $a_{0,q}=0$. From this we obtain the analogous of \eqref{u-to-fi} and \cite[equation (2.27)]{LeZa} for odd Maass cusp forms. By \cite[vol. I, eq. 4.1.(9), p. 130, and vol. II, eq. 8.1.(6), p. 5]{E2} we finally have\footnote{See Remark \ref{notazione-cost}.}
\begin{equation}\label{u-to-fi-odd}
\varphi(t) \doteq t^{1-2q} \, \int_0^t\, \tau^{q-1}\, \HHH_{q-\frac 32} [y\, u_x(i y)] (\tau)\, d\tau = t^{-q}\, \HHH_{q-\frac 12} [u_x(i y)] (t)\, . 
\end{equation}
Notice that the Hankel transform in \eqref{u-to-fi-odd} is absolutely convergent for $\xi>0$ thanks to the rapid decay properties of Maass forms. 

\noindent
We now make use of the involution property of the Hankel transform as in Section \ref{sec:even-good} and repeat the proof of Proposition \ref{ext-to-0}. We first define the ``modified'' inverse of \eqref{u-to-fi-odd}
\begin{definition} \label{def-0-gmf-beta-odd}
For any $q$ with $\Re(q)>0$ and $w\in \C \setminus (1,\infty)$, define the one-parameter family of functions
\begin{equation}\label{fi-to-u-beta-odd}
v_\beta(iy) := \HHH_{q-\frac 12} [\exp_{-\beta}\, \chi_{q}\, \varphi](y)\, , \qquad \Re(\beta)>0
\end{equation}
for functions $\varphi:(0,+\infty) \to \C$ which make the integral converge.
\end{definition}

\noindent
Then we show that if $\varphi$ is an eigenfunction of $(M-N_q)$ then we can put $\beta=0$ in \eqref{fi-to-u-beta-odd}.

\begin{theorem}\label{ext-to-0-odd}
For any $q$ with $\Re(q)>0$, any $w\in \C \setminus (1,\infty)$, and any $\varphi$ as in Proposition \ref{le-fi} with $a_{0,q}=0$, the function $v_\beta(iy)$ can be extended for all $y>0$ as an analytic function of $\beta$ to a small domain containing the origin. Moreover $v_0(iy)$ satisfies
\begin{equation}\label{la-u-odd}
v_0(iy) = w \left[ g(y) - \frac{1}{y^2}\, g\left( \frac 1y \right) \right] \,  ,\qquad \forall\, y>0
\end{equation}
where
$$
g(y)= \HHH_{q-\frac 12}[\exp_{-1}\, \chi_{q}\, \varphi] (y) = \sum_{n=1}^{\infty} (-1)^{n}\, (n+2q-1)\, a_{n-1,q}\, \frac{y^{n+q-2}}{(1+y^{2})^{\frac n2 + \frac q2 + \frac 14}}\, \PPP^{-q+\frac 12}_{n+q-\frac 12} \left( \frac{y}{(1+y^{2})^{\frac 12}} \right)\, ,
$$
and $\{a_{n,q}\}$ is given in \eqref{fi-power-0} with $a_{0,q}=0$.
\end{theorem}

\begin{proof}
Let us fix $y>0$. Using the functional equation $(M-N_q)\varphi = \frac 1w \varphi$, we can write
\begin{equation} \label{est-ux-beta}
v_\beta(iy) = w\, \HHH_{q-\frac 12}[ \exp_{-\beta}\, \chi_{q} M\varphi](y) - w\, \HHH_{q-\frac 12}[ \exp_{-\beta}\, \chi_{q} N_q\varphi](y)
\end{equation}
since the first integral on the right hand side is absolutely convergent. Moreover we can change the order of integration in the second integral, that is
$$
\HHH_{q-\frac 12}[ \exp_{-\beta}\, \chi_{q} N_q\varphi](y) = \int_0^\infty\, J_{q-\frac 12}(ty) \sqrt{ty}\, e^{-\beta t} t^{q} \, \int_0^\infty \, J_{2q-1}(2\sqrt{st}) \left( \frac st \right)^{q-\frac 12}\, e^{-s} \varphi(s) \, ds \ dt =
$$
$$
= \int_0^\infty\, e^{-s} \, s^{q-1} \sqrt{s} \, \varphi(s) \, \int_0^\infty\, J_{q-\frac 12}(ty) \sqrt{ty}\, J_{2q-1}(2\sqrt{st})\, e^{-\beta t}\, t^{\frac 12} dt \ ds
$$
since again the two-variable integral is absolutely convergent. Now we use \cite[vol. I, eq. 4.1.(6), p. 129, and eq. 4.14.(38), p. 186]{E2} to write
\[
\int_0^\infty\, J_{q-\frac 12}(ty) \sqrt{ty}\, J_{2q-1}(2\sqrt{st})\, e^{-\beta t}\, t^{\frac 12} dt = - \sqrt{y}\, \frac{d}{d\beta}\, \int_0^\infty\, J_{q-\frac 12}(ty) \, J_{2q-1}(2\sqrt{st})\, e^{-\beta t}\, dt =
\]
\[
= - \sqrt{y}\, \frac{d}{d\beta}\, \Big[ e^{- \frac{s \beta}{y^2+\beta^2}}\, (y^2 + \beta^2)^{-\frac 12}\, J_{q-\frac 12} \left( \frac{sy}{y^2+\beta^2}\right) \Big]\, ,
\]
and consequently
\[
\HHH_{q-\frac 12}[ \exp_{-\beta}\, \chi_{q} N_q\varphi](y) = - \int_0^\infty\, e^{-s} \, s^{q-1} \sqrt{s} \, \varphi(s) \, \sqrt{y}\, \frac{d}{d\beta}\, \Big[ e^{- \frac{s \beta}{y^2+\beta^2}}\, (y^2 + \beta^2)^{-\frac 12}\, J_{q-\frac 12} \left( \frac{sy}{y^2+\beta^2}\right) \Big]\, ds\, .
\]
Computing all the terms in the previous derivative, we see as in the proof of Theorem \ref{ext-to-0} that all the addends of the integral are absolutely convergent for
\[
\Re \Big( 1 + \frac{\beta}{y^2+\beta^2} \Big) >0\, .
\]
Hence we can again set $\beta=0$ and it turns out that there is only one non-vanishing term, so
\[
\HHH_{q-\frac 12}[ \exp_{-\beta}\, \chi_{q} N_q\varphi]\Big|_{\beta=0}(y) = \frac{1}{y^2}\, \HHH_{q-\frac 12}[ \chi_{q} M\varphi]\left( \frac 1y \right)\, .
\]
So we argue as in Theorem \ref{ext-to-0} and from \eqref{est-ux-beta} we get \eqref{la-u-odd} with
\[
g(y) = \HHH_{q-\frac 12}[ \exp_{-1} \chi_{q} \varphi](y)\, .
\]
The proof is finished as in Theorem \ref{ext-to-0} since we can write $g = \HHH_{q-\frac 12}[ \exp_{-1} \chi_{q-1} \tilde \varphi]$ with $\tilde \varphi (t) = t \varphi(t)$. Hence
\[
\tilde \varphi (t) = \frac{we^{-t}}{1-we^{-t}} \, \sum_{n=1}^{\infty}\, \frac{(-1)^{n-1} a_{n-1,q}\, t^{n}}{\Gamma(n+2q-1)}
\]
and at the end we get
$$
-\frac{1}{y^2}\, g\left( \frac 1y \right) = \sum_{n=1}^{\infty}\, (-1)^{n-1}\, (n+2q-1)\, a_{n-1,q}\, \frac{y^{\frac 12}}{(1+y^{2})^{\frac n2 + \frac q2 + \frac 14}}\, \PPP^{-q+\frac 12}_{n+q-\frac 12} \left( \frac{1}{(1+y^{2})^{\frac 12}} \right)\, .
$$
This finishes the proof.
\end{proof}

\noindent
As in the even case, one can show that the expansion for $v_0(iy)$ obtained in Theorem \ref{ext-to-0-odd} is uniformly convergent on any compact interval of $(0,\infty)$. Moreover we have the following

\begin{corollary}\label{cor-maass-odd}
Letting $w=1$, the function $v_0(iy)$ in \eqref{la-u-odd} is the restriction to the imaginary axis of the x-derivative of an odd Maass cusp form.
\end{corollary}
 \begin{proof}
 It follows from the fundamental theorem of Maass (see \cite[Theorem 2 and Exercise 6, p. 234]{Ter}) that odd Maass cusp forms are uniquely determined by their restriction on the imaginary axis, and correspond to coefficients $\{c_{n,q}\}$ which make the series \eqref{four-exp-odd} satisfy $u_x(iy)= - y^2\, u_x(i\frac 1y)$. 
 
 \noindent
By definition we have that the function $v_0(iy)$ in \eqref{la-u-odd} satisfies $v_0(iy)= - y^2\, v_0(i\frac 1y)$. Then the proof is finished by using \eqref{nuova-varfi-odd} and \cite[Chap. II, Section 3]{LeZa}, to show that \eqref{u-to-fi-odd} is a bijection between odd Maass cusp forms and the eigenfunctions of $M-N_q$ as in Proposition \ref{le-fi}.
 \end{proof}

\section{From gpf to Fourier coefficients of Maass forms}\label{sec-courier}

We now use equations \eqref{u-to-fi}, \eqref{eis-to-fi} and \eqref{u-to-fi-odd}, to obtain relations between the coefficients of the power series expansions of the eigenfunctions $\varphi$ introduced in Proposition \ref{le-fi}, and the Fourier coefficients of the Maass forms. In the case of the non-holomorphic Eisenstein series, this approach brings interesting results for the divisor function $\sigma_\alpha(n)$. The results are summarized in Theorem B in the Introduction.

\noindent
The main equality we need is the symmetric Hankel transform of the Bessel functions $K_\nu$, which is given in \cite[vol. II, eq. 8.13.(2), p. 63]{E2}, namely
\begin{equation}\label{h-trasf}
\HHH_\nu \left[ y^{\frac 12}\, K_\nu(ay)\right] (t) = a^{-\nu}\, \frac{t^{\nu+\frac 12}}{t^2+a^2}\, ,\qquad \Re(a)>0\, ,\, \Re(\nu)> -1\, .
\end{equation}

\subsection{Maass cusp forms}

Let $u(x+iy)$ be an even Maass cusp form, we can then use its Fourier series expansion
\[
u(x+i y) = y^{\frac 12}\, \sum_{n\ge 1}\, c_{n,q}\, K_{q-\frac 12}(2\pi ny)\, \cos(2\pi nx)
\]
in 
\[
\varphi(t) \doteq t^{1-q}\, \HHH_{q-\frac 12} [u(i y)](t)\, ,
\]
from which using term-by-term \eqref{h-trasf} with $a=2\pi n$ and $\nu=q-\frac 12$, one obtains (see \cite[equation (2.28)]{LeZa})
\begin{equation}\label{nuova-varfi}
\varphi(t) \doteq \sum_{n\ge 1}\, n^{\frac 12 -q}\, c_{n,q}\, \frac{t}{t^2+(2\pi n)^2} = \sum_{k=0}^\infty\, \frac{(-1)^k}{(4\, \pi^2)^{k+1}}\, L_u\Big(q+2k+\frac 32\Big)\, t^{2k+1}
\end{equation}
where $L_u(\rho)$ is the Dirichlet $L$-series associated to $u$, namely
\begin{equation}\label{ell-series}
L_u(\rho) : = \sum_{n\ge 1}\, c_{n,q}\, n^{-\rho}\, .
\end{equation}
This shows that $\varphi$ is an odd function, and gives a relation between the coefficients $\{a_{k,q}\}$ in
\[
\varphi(t) = \frac{e^{-t}}{1-e^{-t}} \, \sum_{k=1}^{\infty}\, \frac{(-1)^{k} a_{k,q}\, t^{k}}{\Gamma(k+2q)} 
\]
and $\{c_{n,q}\}$. Indeed, recalling that
\[
\frac{e^{-t}}{1-e^{-t}} = \sum_{i\ge 0}\, B_i \, \frac{t^{i-1}}{i!}
\]
where $\{B_i\}$ are the Bernoulli numbers, we find\footnote{See Remark \ref{notazione-cost}.}
\begin{equation}\label{relaz-coefficienti}
\sum_{i\ge -1,\, j\ge 1,\, i+j=n}\, (-1)^j \, \frac{B_{i+1}\, a_{j,q}}{(i+1)!\, \Gamma(j+2q)} \doteq \left\{
\begin{array}{ll}
\frac{(-1)^{\frac{n-1}{2}}}{(2\pi)^{n+1}}\, L_u\Big(q+n+\frac 12\Big)\, , & \text{if $n$ is odd}\\[0.2cm]
0\, , & \text{if $n$ is even}
\end{array} \right.
\end{equation}

\noindent
The equality \eqref{nuova-varfi} can be further used to find an even more direct expression for the $c_{n,q}$ in terms of the $\{a_{n,q}\}$. Following a suggestion first given in \cite{Le} (see also \cite{LeZa}),
let us introduce the {\sl interpolating function} $g_q$, defined as
\begin{equation}\label{gi}
g_q(t)= \sum_{k=1}^{\infty}\, \frac{ a_{k,q} t^{k}}{\Gamma(k+2q)} \equiv \sum_{k=1}^{\infty}\, \frac{ \beta_{k,q} t^{k}}{k!}
\end{equation}
where we have set
$$
\beta_{k,q} = \frac {k! }{\Gamma(k+2q)}\, a_{k,q}
$$
One readily sees that $g_q$ is entire of exponential type and, moreover, we find
\[
g_q(-t)=\sum_{k=1}^{\infty}\, \frac{ (-1)^k \beta_{k,q} t^{k}}{k!} =(e^t-1)\varphi(t)
\]
Now, since (the meromorphic continuation of) $\varphi$ is odd, we have
$g_q(t)=(1-e^{-t})\varphi(t)$. We thus see that $g_q$ satisfies the functional equation
$g_q(-t) = e^t g_q(t)$. 
The name of the function $g_q$ comes from the following fact: taking the limit $t\to \pm 2\pi i n$ in (\ref{gi}), using the first identity of (\ref{nuova-varfi}) and observing that 
$$
\lim_{t\to \pm 2\pi i n} \frac {t(e^t-1)}{t^2+(2\pi n)^2} =\lim_{t\to \pm 2\pi i n} \frac {t(e^t-1)}{(t+2\pi in)(t-2\pi i n)} =\frac 12
$$
we obtain the following formula for the Fourier coefficients of $u$:
\begin{equation}\label{formu}
c_{n,q} \doteq 2 \, n^{q-\frac 12}g_q(\pm 2\pi i n) \quad , \quad n\geq 1
\end{equation}
In order to better understand the consequences of this formula we have to study the behavior  of the entire function $g_q$ on the imaginary axis.
For the moment we just put the above formula in a more explicit form, using (\ref{gi}) and (\ref{formu}), 
\[
c_{n,q} \doteq 2 \, n^{q-\frac 12}\sum_{k=1}^{\infty}\, \frac{\beta_{k,q} }{k!} (\pm 2\pi i n)^{k} \quad , \quad n\geq 1
\]
The symmetry with respect to the change of sign yields 
\[
\sum_{k\geq 0}\, \frac{\beta_{2k+1,q}}{(2k+1)!}\, (2\pi i n)^{2k+1}=0  \quad , \quad \forall n\ge 1
\]
so that we can finally write
\begin{equation}\label{formu2}
c_{n,q} \doteq 2 \, n^{q-\frac 12}\sum_{k=1}^{\infty}\, \frac{(-1)^{k} \beta_{2k,q} }{(2k)!} (2\pi n)^{2k} = 2 \, n^{q-\frac 12}\sum_{k=1}^{\infty}\, \frac{(-1)^{k} a_{2k,q} }{\Gamma(2k+2q)} (2\pi n)^{2k} \quad , \quad n\geq 1
\end{equation}
and finish the proof of Theorem B for even Maass cusp forms. 

\noindent
Finally, we remark that the functional equation $g_q(-t) = e^t g_q(t)$ gives
$$
\sum_{k=1}^{\infty}  (-1)^k \frac {\beta_{k,q}}{k!}  t^{k} = \sum_{k=1}^{\infty}\left(\sum_{\ell=1}^{k}  \frac {\beta_{\ell,q}}{\ell!(k-\ell)!} \right)t^{k}
$$
and therefore
\[
\sum_{\ell=1}^{k}  \binom{k}{\ell} \beta_{\ell,q}=(-1)^k \beta_{k,q}\quad , \quad k\geq 1
\]
which (for $k$ even) is akin to the recursive property of the Bernoulli numbers.

\noindent
In the odd case, everything works similarly. Using \eqref{four-exp-odd} and integrating term-by-term we get the analogous of \eqref{nuova-varfi}, namely
\begin{equation}\label{nuova-varfi-odd}
\begin{aligned}
\varphi(t) \doteq &\ t^{-q} \sum_{n\ge 1}\, n\, c_{n,q}\, \HHH_{q-\frac 12} \Big[\chi_{\frac 12}(y)\, K_{q-\frac 12}(2\pi n y)\Big] (t) \\[0.2cm]
\doteq &\ \sum_{n\ge 1}\, n^{-q-\frac 12}\, c_{n,q} \, \Big(\frac{1}{(\frac{t}{2\pi n})^2+1} -1\Big) + L_u\Big(q+\frac 12 \Big)  =  \sum_{k=0}^\infty\, \frac{(-1)^k}{(2\pi)^{2k}}\, L_u\Big(q+2k +\frac 12 \Big)\, t^{2k}
\end{aligned}
\end{equation}
where in the second line we have used \cite[vol. II, eq. 8.13.(2), p. 63]{E2}, the definition of the $L$-series $L_u$ in \eqref{ell-series} and their analytic extensions. It follows that $\varphi$ is an even function and the analogous of \eqref{relaz-coefficienti} and \eqref{formu2} are immediate.

\subsection{Non-holomorphic Eisenstein series}

Analogous computations can be performed starting from \eqref{eis-to-fi} and applying \eqref{h-trasf} term-by-term. In this case however the coefficients are known, hence we obtain explicit equalities. Letting
\[
\tilde \varphi (t) := \frac{1}{\Gamma(2q)}\, \frac{1}{e^{t}-1}\, \sum_{k\ge 0} A_{k,q}\, \frac{t^{k}}{k!}
\]
with
\[
A_{k,q} = \frac{1}{k+1}\, \sum_{i=0}^{k}\, {k+1\choose i}\, B_{i}\, \zeta(2q-1+i)\, , \qquad k\ge 0\, ,
\]
where $B_i$ are the Bernoulli numbers, we first rewrite it with
\[
\tilde A_{k,q} := A_{k,q} - \frac{\zeta(2q-1)}{k+1} + \frac{\zeta(2q)}{2} = \frac{1}{k+1}\, \sum_{i=2}^{k}\, {k+1\choose i}\, B_{i}\, \zeta(2q-1+i)\, \qquad k\ge 2\, ,
\]
from which we obtain
\[
\tilde \varphi(t) = \frac{\zeta(2q-1)}{\Gamma(2q)}\, \frac 1t - \frac{\zeta(2q)}{2\, \Gamma(2q)} + \frac{1}{\Gamma(2q)}\, \frac{1}{e^{t}-1}\, \sum_{k\ge 2} \tilde A_{k,q}\, \frac{t^{k}}{k!}\, .
\]
Then writing
\[
\tilde E(iy,q) =  \frac{\pi^{\frac 12}\, \Gamma(q-\frac 12)}{\Gamma(q)}\, \zeta(2q-1)\, y^{1-q} - \zeta(2q)\, y^{-q} + \frac{4\, \pi^q}{\Gamma(q)}\, y^{\frac 12}\, \sum_{n\ge 1}\, n^{\frac 12-q}\, \sigma_{2q-1}(n) \, K_{q-\frac 12}(2\pi ny)
\]
where $\sigma_\alpha(n) := \sum_{d|n}\, d^\alpha$ is the divisor function, equality \eqref{eis-to-fi} gives
\[
\tilde \varphi(t) = \frac{\zeta(2q-1)}{\Gamma(2q)}\, \frac 1t - \frac{\zeta(2q)}{2\, \Gamma(2q)} + \frac{2}{\Gamma(2q)}\, \sum_{n\ge 1}\, n^{1-2q}\, \sigma_{2q-1}(n)\, \frac{t}{t^2+(2\pi n)^2}\, ,
\]
where we have used \cite[vol. II, eq. 8.5.(7), p. 22]{E2}, \eqref{h-trasf} and the equality $2^{1-2q}\, \pi^{\frac 12}\, \Gamma(2q) = \Gamma(q)\, \Gamma(q+\frac 12)$, hence
\begin{equation}\label{nuova-varfi-eis}
\frac{1}{e^{t}-1}\, \sum_{k\ge 2} \tilde A_{k,q}\, \frac{t^{k}}{k!} = 2\, \sum_{n\ge 1}\, n^{1-2q}\, \sigma_{2q-1}(n)\, \frac{t}{t^2+(2\pi n)^2} = 2\, \sum_{k=0}^\infty\, \frac{(-1)^k}{(4\, \pi^2)^{k+1}}\, L_\sigma\Big(2q+2k+1\Big)\, t^{2k+1}
\end{equation}
where now $L_\sigma(\rho)$ is the Dirichlet $L$-series
\[
L_\sigma(\rho) : = \sum_{n\ge 1}\, \sigma_{2q-1}(n)\, n^{-\rho}\, .
\]
Arguing as above, from \eqref{nuova-varfi-eis} we obtain the analogous of \eqref{relaz-coefficienti} and \eqref{formu2}
\begin{equation}\label{relaz-coeff-eis}
\sum_{i\ge -1,\, j\ge 1,\, i+j=n}\, \frac{B_{i+1}\, \tilde A_{j,q}}{(i+1)!\, j!} = \left\{
\begin{array}{ll}
\frac{(-1)^{\frac{n-1}{2}}}{2^n \, \pi^{n+1}}\, L_\sigma\Big(2q+n\Big)\, , & \text{if $n$ is odd}\\[0.2cm]
0\, , & \text{if $n$ is even}
\end{array} \right.
\end{equation}
and
\begin{equation}\label{new-div-fun}
\sigma_{2q-1}(n) = n^{2q-1}\, \sum_{k=1}^{\infty}\, \frac{(-1)^k\, \tilde A_{2k,q}}{(2k)!}\, (2\pi n)^{2k}\, , \qquad \forall\, n\ge 1\, .
\end{equation}
for all $q$ with $\xi >0$. The convergence on the right hand side of \eqref{new-div-fun} is absolutely since $|\tilde A_{k,q}| \le \text{const}\, k^2\, \zeta(2\xi+1)$, as we prove in Theorem \ref{zagier-funct-ser}. This finishes the proof of Theorem B.

\noindent
From \eqref{new-div-fun} we can obtain a new formulation also for the partial sums of $\frac{\sigma_{2q-1}(n)}{n^{2q-1}}$. Using 
\[
\sum_{n=1}^N\, n^{2k} = \frac{1}{2k+1} ( B_{2k+1}(N) - B_{2k+1} )
\] 
where $B_n(x)$ are the Bernoulli polynomials, we get
\[
\sum_{n=1}^N\, \frac{\sigma_{2q-1}(n)}{n^{2q-1}} = \sum_{k=1}^\infty\, \frac{(-1)^k\, (2\pi)^{2k}\, \tilde A_{2k,q}}{(2k+1)!}\, \Big(B_{2k+1}(N) - B_{2k+1}\Big)\, .
\]
Notice that the function
\[
F(x) := \sum_{k=1}^\infty\, \frac{(-1)^k\, (2\pi)^{2k}\, \tilde A_{2k,q}}{(2k+1)!}\, \Big(B_{2k+1}(x) - B_{2k+1}\Big)
\]
is uniformly convergent on any compact interval of the real line. 

\begin{rem}
Equation \eqref{relaz-coeff-eis} is equivalent to the formulas $\zeta(2k) = (-1)^{k+1} \, \frac{B_{2k} (2\pi)^{2k}}{2 (2k)!}$ and $L_\sigma(2q+2k+1) = \zeta(2q+2k+1)\, \zeta(2k+2)$, with $k\ge 0$. In particular assuming one of the two formulas and \eqref{relaz-coeff-eis}, one obtains the other formula.
\end{rem}

\begin{rem} \label{rem-ram}
By using the expansion
\[
\tilde A_{2k,q} =  \sum_{\ell\ge 2}\, \frac{S_{2k}(\ell) - \frac{1}{2k+1}\, \ell^{2k+1} - \frac 12 \, \ell^{2k}}{\ell^{2q+2k}} - \frac{1}{2k+1} + \frac 12
\]
for $\Re(q)>1$ one can change the order of summation in \eqref{new-div-fun}, and use the expression $c_r(n) = \sum_{(i,r)=1\, , \, i\le r}\, \cos\left( 2\pi \frac i r n\right)$ for the Ramanujan's sums to write
\[
\frac{\sigma_{2q-1}}{n^{2q-1}} =  \sum_{k=1}^{\infty}\, \frac{(-1)^k\, \tilde A_{2k,q}}{(2k)!}\, (2\pi n)^{2k} = T_1 + T_2 + T_3 + T_4 + T_5
\]
where
\[
T_2 = - \sum_{k=1}^{\infty}\, \left( \frac{1}{2k+1}\, \sum_{\ell\ge 2}\, \frac{\ell^{2k+1}}{\ell^{2q+2k}} \right) \, \frac{(-1)^k}{(2k)!}\, (2\pi n)^{2k} = - \Big( \zeta(2q-1) -1 \Big)\, \frac{\sin t -t}{t}\Big|_{t=2\pi n} = \zeta(2q-1) -1
\]
\[
T_3 = - \sum_{k=1}^{\infty}\, \left( \frac{1}{2}\, \sum_{\ell\ge 2}\, \frac{\ell^{2k}}{\ell^{2q+2k}} \right) \, \frac{(-1)^k}{(2k)!}\, (2\pi n)^{2k} = - \frac 12 \Big(\zeta(2q)-1 \Big) \, (\cos t - 1)\Big|_{t=2\pi n} = 0
\]
\[
T_4 = - \sum_{k=1}^{\infty}\, \frac{1}{2k+1}\, \frac{(-1)^k}{(2k)!}\, (2\pi n)^{2k} = - \frac{\sin t -t}{t}\Big|_{t=2\pi n} = 1
\]
\[
T_5 = \sum_{k=1}^{\infty}\, \frac{1}{2}\, \frac{(-1)^k}{(2k)!}\, (2\pi n)^{2k} = (\cos t - 1)\Big|_{t=2\pi n} = 0
\]
and for the first term one gets
\[
T_1 = \sum_{k=1}^{\infty}\, \left(  \sum_{\ell\ge 2}\, \sum_{j=1}^\ell\, \frac{j^{2k}}{\ell^{2q+2k}} \right)\, \frac{(-1)^k}{(2k)!}\, (2\pi n)^{2k} = \sum_{\ell\ge 2}\, \frac{1}{\ell^{2q}}\, \sum_{j=1}^\ell\, \left( \sum_{k=1}^{\infty}\, \frac{(-1)^k}{(2k)!} \left( 2\pi\, \frac j \ell\, n\right)^{2k} \right) =
\]
\[
= \sum_{\ell\ge 2}\, \frac{1}{\ell^{2q}}\, \sum_{j=1}^\ell\, \Big( \cos \left( 2\pi\, \frac j \ell\, n\right) -1 \Big) = -\zeta(2q-1)+1 + \sum_{\ell\ge 2}\, \frac{1}{\ell^{2q}}\, \sum_{j=1}^\ell\, \cos \left( 2\pi\, \frac j \ell\, n\right) = 
\]
\[
= -\zeta(2q-1)+ \sum_{\ell\ge 1}\, \frac{1}{\ell^{2q}}\, \sum_{j=1}^\ell\, \cos \left( 2\pi\, \frac j \ell\, n\right) = -\zeta(2q-1)+ \sum_{\ell\ge 1}\, \frac{1}{\ell^{2q}}\, \sum_{r|\ell}\, c_r(n) = 
\]
\[
= -\zeta(2q-1)+ \zeta(2q)\, \sum_{\ell\ge 1}\, \frac{c_\ell(n)}{\ell^{2q}}
\]
where in the last equality we have used Dirichlet multiplication to write $\sum_{r|\ell}\, c_r(n) = (c(n) \star u)(\ell)$, where $u(k)=1$ for all $k\ge 1$. We have thus obtained Ramanujan's expansion for the divisor function for $\Re(q)>1$ (see \cite{ram}), namely
\[
\frac{\sigma_{2q-1}}{n^{2q-1}} = \sum_{k=1}^{\infty}\, \frac{(-1)^k\, \tilde A_{2k,q}}{(2k)!}\, (2\pi n)^{2k} = \zeta(2q)\, \sum_{\ell\ge 1}\, \frac{c_\ell(n)}{\ell^{2q}}\, ,
\]
We remark that Ramanujan expansion is known to hold only for $\Re(q)> \frac 12$, where the Dirichlet series is absolutely convergent, for $q=\frac 12$ where the convergence of the Dirichlet series is equivalent to the Prime Number Theorem, and can be extended to $\Re(q) \in \left( \frac 14, \frac 12\right]$ assuming Riemann Hypothesis. 
On the other hand, expansion
\eqref{new-div-fun} holds for $\Re(q)>0$.
\end{rem}

\section{Power series expansions for Maass forms on the imaginary axis} \label{sec-power}

Equations \eqref{legendre-exp} and \eqref{fi-for-psi} provide series expansions for Maass forms in terms of the Legendre functions $\PPP_{\nu}^{\mu}$. Moreover, in the case of non-cusp forms we have explicit expressions for the coefficients $b_{n,q}$ of the series. We now use properties of the Legendre functions to obtain different expansions in terms of rational functions.

\noindent
The functions involved are $\PPP_{\nu}^{\mu}$ with $\nu= n+q-\frac 12$ and $\mu=-q+\frac 12$, so that $\mu + \nu = n$ is an integer. Consequently, using \cite[eq. 14.3.11, p. 354]{nist}, we get
\[
\PPP_{n+q-\frac 12}^{-q+\frac 12}(t) = \left\{
\begin{array}{ll}
(-1)^{\frac n2}\, \frac{2^{-q+\frac 12} \Gamma(\frac{n+1}{2})}{\Gamma(\frac{n+1}{2}+q)} \, (1-t^{2})^{\frac q2 -\frac 14}\, \hyperbold{2}{1}\left( -\frac n2,\, \frac n2 +q;\, \frac 12;\, t^{2}\right), & \text{if $n$ is even}\\[0.3cm]
(-1)^{\frac{n-1}{2}}\, \frac{2^{-q+\frac 12} \Gamma(\frac{n}{2}+1)}{\Gamma(\frac{n}{2}+q)} \, t\, (1-t^{2})^{\frac q2 -\frac 14}\, \hyperbold{2}{1}\left( -\frac{n-1}{2},\, \frac{n+1}{2} +q;\, \frac 32;\, t^{2}\right), & \text{if $n$ is odd}
\end{array}
\right.
\] 
where $\hyperbold{2}{1}$ is the scaled hypergeometric function. Moreover, since the first variable of $\hyperbold{2}{1}$ is in both case a non-positive integer, then the hypergeometric function is a polynomial in $t^{2}$, more precisely for $k\in \N$ and $c\not=0,-1,-2,\dots$, it holds
$$
\hyperbold{2}{1}(-k,\, b;\, c;\, t^{2}) = \sum_{j=0}^{k}\, (-1)^{j}\, {k\choose j}\, \frac{\Gamma(b+j)}{\Gamma(b)\, \Gamma(c+j)}\, t^{2j}\, .
$$
It follows that for the terms in \eqref{legendre-exp} and \eqref{fi-for-psi}, we get
$$
\frac{y^{\frac 12}\, \PPP^{-q+\frac 12}_{n+q-\frac 12} \left( \frac{1}{(1+y^{2})^{\frac 12}} \right) + y^{n+q}\, \PPP^{-q+\frac 12}_{n+q-\frac 12} \left( \frac{y}{(1+y^{2})^{\frac 12}} \right)}{(1+y^{2})^{\frac n2 + \frac q2 + \frac 14}} =
$$
\begin{equation}\label{formulina}
\left\{
\begin{array}{ll}
(-1)^{\frac n2}\, \frac{2^{-q+\frac 12} \Gamma(\frac{n+1}{2})}{\Gamma(\frac{n+1}{2}+q)\, \Gamma(\frac n2 + q)} \,  \left( \frac y{1+y^2}\right)^q \, \sum_{j=0}^{n/2}\, (-1)^{j}\, {n/2\choose j}\, \frac{\Gamma(\frac n2 + q+j)}{\Gamma(\frac 12+j)}\, \frac{1+y^{n+2j}}{(1+y^{2})^{\frac n2 +j}}\, , & \text{if $n$ is even}\\[0.3cm]
(-1)^{\frac{n-1}2}\, \frac{2^{-q+\frac 12} \Gamma(\frac{n}{2}+1)}{\Gamma(\frac{n+1}{2}+q)\, \Gamma(\frac n2 + q)} \,  \left( \frac y{1+y^2}\right)^q \, \sum_{j=0}^{(n-1)/2}\, (-1)^{j}\, {(n-1)/2\choose j}\, \frac{\Gamma(\frac{n+1}2 + q+j)}{\Gamma(\frac 32+j)}\, \frac{1+y^{n+1+2j}}{(1+y^{2})^{\frac{n+1}2 +j}}\, , & \text{if $n$ is odd}
\end{array}
\right.
\end{equation}
Notice that each term of the finite sums is invariant with respect to the transformation $y\mapsto \frac 1y$. 

\noindent
We now substitute \eqref{formulina} into \eqref{legendre-exp} and \eqref{fi-for-psi} and get, with $\alpha_{n,q}$ being equal to $a_{n,q}$ and $b_{n,q}$ respectively,
\[
\sum_{n=0}^\infty\, (-1)^n \alpha_{n,q}\, \ \frac{y^{\frac 12}\, \PPP^{-q+\frac 12}_{n+q-\frac 12} \left( \frac{1}{(1+y^{2})^{\frac 12}} \right) + y^{n+q}\, \PPP^{-q+\frac 12}_{n+q-\frac 12} \left( \frac{y}{(1+y^{2})^{\frac 12}} \right)}{(1+y^{2})^{\frac n2 + \frac q2 + \frac 14}} =
\]
\[
= \sum_{k=0}^\infty\, \alpha_{2k,q}\ (-1)^k\, \frac{2^{-q+\frac 12} \Gamma(\frac{2k+1}{2})}{\Gamma(\frac{2k+1}{2}+q)\, \Gamma(k + q)} \,  \left( \frac y{1+y^2}\right)^q \, \sum_{j=0}^{k}\, (-1)^{j}\, {k\choose j}\, \frac{\Gamma(k + q+j)}{\Gamma(\frac 12+j)}\, \frac{1+y^{2(k+j)}}{(1+y^{2})^{k +j}} +
\]
\[
+ \sum_{h=0}^\infty\, (-1)\,  \alpha_{2h+1,q}\ (-1)^{h}\, \frac{2^{-q+\frac 12} \Gamma(\frac{2h+1}{2}+1)}{\Gamma(h+1+q)\, \Gamma(\frac {2h+1}{2} + q)} \,  \left( \frac y{1+y^2}\right)^q \, \sum_{j=0}^{h}\, (-1)^{j}\, {h\choose j}\, \frac{\Gamma(h+ q+j+1)}{\Gamma(\frac 32+j)}\, \frac{1+y^{2(h+1+j)}}{(1+y^{2})^{h+1+j}}
\]
where we have split the sum in one with even indices and one with odd indices. At this point we can group together all the coefficients multiplying terms of the form $\frac{1+y^{2s}}{(1+y^{2})^{s}}$ and get
\[
\sum_{n=0}^\infty\, (-1)^n \alpha_{n,q} \ \frac{y^{\frac 12}\, \PPP^{-q+\frac 12}_{n+q-\frac 12} \left( \frac{1}{(1+y^{2})^{\frac 12}} \right) + y^{n+q}\, \PPP^{-q+\frac 12}_{n+q-\frac 12} \left( \frac{y}{(1+y^{2})^{\frac 12}} \right)}{(1+y^{2})^{\frac n2 + \frac q2 + \frac 14}} = 2^{-q+\frac 12} \left( \frac y{1+y^2}\right)^q \sum_{s=0}^{\infty} (-1)^s\, \eta_{s,q}\, \frac{1+y^{2s}}{(1+y^{2})^{s}}
\]
where the coefficients $\eta_{s,q}$ are a given by a finite sum. In particular
\begin{equation}\label{davvero}
\eta_{s,q}:= \left\{ \begin{array}{ll}
\sum_{i=0}^{s}\, \alpha_{s+i,q}\, \gamma_{s+i,q}\, \delta_{s+i, \frac s2 - \lfloor \frac{i+1}{2} \rfloor,q}\, , & \text{ $s$ even}\\[0.2cm]
\sum_{i=0}^{s}\, \alpha_{s+i,q}\, \gamma_{s+i,q}\, \delta_{s+i, \frac{s-1}{2} - \lfloor \frac{i}{2} \rfloor,q}\, , & \text{ $s$ odd}
\end{array} \right.
\end{equation}
with 
\[
\gamma_{2k,q} =  \frac{ \Gamma(\frac{2k+1}{2})}{\Gamma(\frac{2k+1}{2}+q)\, \Gamma(k + q)} \quad \text{and} \quad  \gamma_{2k+1,q} =  \frac{\Gamma(\frac{2k+1}{2}+1)}{\Gamma(\frac{2k+1}{2}+q)\, \Gamma(k + q+1)}
\]
\[
\delta_{2k,j,q} =  {k\choose j}\, \frac{\Gamma(k + q+j)}{\Gamma(\frac 12+j)} \quad \text{and} \quad \delta_{2k+1,j,q} =  {k\choose j}\, \frac{\Gamma(k + q+j+1)}{\Gamma(\frac 32+j)}
\]
Introducing the notation $\beta_{n,q} := (-1)^n\, \frac{n!\, \Gamma(2q)}{\Gamma(n+2q)}\, \alpha_{n,q}$ in \eqref{davvero} we get
\begin{equation} \label{davverissimo}
\eta_{s,q} = (-1)^s\, 2^s\, \frac{\Gamma(s+q)}{s!\, \Gamma(q+\frac 12)\, \Gamma(q)} \ \sum_{i=0}^s\, (-1)^i\, 2^{-i}\, \beta_{s+i,q}\, {s\choose i}\, .
\end{equation}
We have thus proved
\begin{proposition}\label{prop:cusp-rat}
An even Maass cusp form with eigevalue $q(1-q)$, can be formally written when restricted to the imaginary axis as
\[
u(iy) = 2^{-q+\frac 12} \left( \frac y{1+y^2}\right)^q \sum_{s=0}^{\infty} (-1)^s\, \eta_{s,q}\, \frac{1+y^{2s}}{(1+y^{2})^{s}}
\]
with $\eta_{s,q}$ as in \eqref{davverissimo} and $\beta_{n,q} = (-1)^n\, \frac{n!\, \Gamma(2q)}{\Gamma(n+2q)}\, a_{n,q}$, where $\{a_{n,q}\}$ is given in \eqref{fi-power-b}.
\end{proposition}

When writing the same expansion for the non-cusp forms $U(iy)$ as in \eqref{fi-for-psi}, we can use the explicit expression for the coefficient $\{b_{n,q}\}$, which are defined in terms of the $\{A_{n,q}\}$ of Theorem \ref{zagier-funct-ser}. We first get
\begin{equation}\label{eis-series}
U(iy) = \zeta(2q)\, \Big( y^{q} + y^{-q} \Big) + 2\left( \frac{y}{1+y^2} \right)^q \left[ - \zeta(2q)\,  + \sum_{s=0}^\infty\, 2^s\, \frac{\Gamma(s+q)}{s!\, \Gamma(q)} \left( \sum_{i=0}^s\, {s\choose i}\, \frac{(-1)^i\, A_{s+i,q}}{2^i} \right) \, \frac{1+y^{2s}}{(1+y^{2})^{s}} \right]
\end{equation}
then recalling
\[
A_{n,q} = \frac{1}{n+1}\, \sum_{\ell=0}^{n}\, {n+1 \choose \ell}\, B_\ell\, \zeta(2q-1+\ell)\, ,
\]
we obtain for $s\ge 1$
\[
\sum_{i=0}^s\, {s\choose i}\, \frac{(-1)^i\, A_{s+i,q}}{2^i} = \sum_{i=0}^s\, \sum_{\ell=0}^{s+i}\, B_\ell\, \zeta(2q-1+\ell)\, {s+i+1\choose \ell}\, {s\choose i}\, \frac{(-1)^i}{2^i\, (s+i+1)} =
\]
\[
= \sum_{\ell=0}^s\, \left( \sum_{i=0}^s\, {s+i+1\choose \ell}\, {s\choose i}\, \frac{(-1)^i}{2^i\, (s+i+1)} \right)\, B_\ell\, \zeta(2q-1+\ell) +
\]
\[
+  \sum_{\ell=s+1}^{2s}\, \left( \sum_{i=\ell-s}^s\, {s+i+1\choose \ell}\, {s\choose i}\, \frac{(-1)^i}{2^i\, (s+i+1)} \right)\, B_\ell\, \zeta(2q-1+\ell)
\]
Moreover for $\ell\ge 2$
\[
\sum_{i=0}^s\, {s+i+1\choose \ell}\, {s\choose i}\, \frac{(-1)^i}{2^i\, (s+i+1)} = \frac 1 \ell \, {s\choose \ell-1}\, \sum_{i=0}^s\, {s\choose i}\, \frac{\Gamma(s+i+1)\, \Gamma(s+2-\ell)}{\Gamma(s+1)\, \Gamma(s+i+2-\ell)}\, \left( - \frac 12\right)^i =
\]
\[
= \frac 1 \ell \, {s\choose \ell-1}\, \hyper{2}{1}\Big(-s, s+1; s+2-\ell; \frac 12\Big) = \frac 1 \ell \, {s\choose \ell-1}\, \frac{\pi^{\frac 12}\,\, \Gamma(s+2-\ell)}{2^{s+1-\ell}\, \Gamma(1-\frac \ell 2)\, \Gamma(s+\frac{3-\ell}{2})}
\]
where in the last equality we have used \cite[eq. 15.4.30, p. 387]{nist}. Notice that the last expression vanishes for $\ell$ even because of the poles of the term $\Gamma(1-\frac \ell 2)$, and since $B_\ell =0$ for $\ell \ge 2$ and odd, we obtain
\[
\sum_{\ell=0}^s\, \left( \sum_{i=0}^s\, {s+i+1\choose \ell}\, {s\choose i}\, \frac{(-1)^i}{2^i\, (s+i+1)} \right)\, B_\ell\, \zeta(2q-1+\ell) =
\]
\[ =  B_0 \, \zeta(2q-1)\, \sum_{i=0}^s\, {s\choose i}\, \frac{(-1)^i}{2^i\, (s+i+1)} + B_1\, \zeta(2q)\, \sum_{i=0}^s\, {s\choose i}\, \frac{(-1)^i}{2^i}= \frac{\pi^{\frac 12}}{2^{s+1}}\, \frac{\Gamma(s+1)}{\Gamma(s+\frac 32)}\, \zeta(2q-1) - \frac{1}{2^{s+1}}\, \zeta(2q)\, .
\]
Similarly, letting $k=\ell -s\ge 1$, we have
\[
\sum_{i=k}^s\, {s+i+1\choose s+k}\, {s\choose i}\, \frac{(-1)^i}{2^i\, (s+i+1)} = \frac{(-1)^k}{2^{k-1}\, (s+k)}\, {s\choose k-1}\, - \frac{(-1)^k\, \pi^{\frac 12}\, (s+k+1)}{(s+k)\, (s-k+1)}\, {s\choose k}\, \frac{\Gamma(k+1)}{\Gamma(\frac{k-s}{2})\, \Gamma(\frac{s+k+3}{2})} 
\]
whose second term vanishes for $s-k$ even, that is for $\ell = s-k +2k$ even, and since $B_\ell =0$ for $\ell \ge 2$ and odd, we obtain for $s\ge 1$
\[
\sum_{\ell=s+1}^{2s}\, \left( \sum_{i=\ell-s}^s\, {s+i+1\choose \ell}\, {s\choose i}\, \frac{(-1)^i}{2^i\, (s+i+1)} \right)\, B_\ell\, \zeta(2q-1+\ell) = 
\]
\[ 
= \sum_{k=1}^{s}\, \left( \sum_{i=k}^s\, {s+i+1\choose s+k}\, {s\choose i}\, \frac{(-1)^i}{2^i\, (s+i+1)} \right)\, B_{s+k}\, \zeta(2q-1+s+k) =
\]
\[
= \sum_{k=1}^{s}\, \frac{(-1)^k}{2^{k-1}\, (s+k)}\, {s\choose k-1}\, B_{s+k}\, \zeta(2q-1+s+k)\, .
\]
We have thus shown that for $s\ge 1$
\begin{equation}\label{ultimi-coeff}
\begin{aligned}
\sum_{i=0}^s\, {s\choose i}\, \frac{(-1)^i\, A_{s+i,q}}{2^i} = & \frac{\pi^{\frac 12}}{2^{s+1}}\, \frac{\Gamma(s+1)}{\Gamma(s+\frac 32)}\, \zeta(2q-1) - \frac{1}{2^{s+1}}\, \zeta(2q) + \\[0.3cm]
& + \sum_{k=1}^{s}\, \frac{(-1)^k}{2^{k-1}\, (s+k)}\, {s\choose k-1}\, B_{s+k}\, \zeta(2q-1+s+k)
\end{aligned}
\end{equation}
Substituting \eqref{ultimi-coeff} in \eqref{eis-series}, and using the identities
\[
\sum_{s=0}^\infty\, \frac{\Gamma(s+q)}{s!\, \Gamma(q)}\, z^s = (1-z)^{-q}\quad \text{and} \quad \sum_{s=0}^\infty\, \frac{\Gamma(s+q)}{\Gamma(s+\frac 32)\, \Gamma(q)}\, z^s = \frac{2}{\pi^{\frac 12}}\, \hyper{2}{1}\left(1,q\, ;\, \frac 32\, ;\, z\right)\, ,
\]
we obtain that the Eisenstein series $E(iy,q)$ defined in \eqref{eisen-q} can be formally written as
\begin{equation} \label{forma-eis}
\begin{aligned}
E(iy,q) = &\,  2 \, \left(\frac{y}{1+y^2}\right)^q \, \left[ \zeta(2q-1)\, \hyper{2}{1}\left(1,q\, ;\, \frac 32\, ;\, \frac{1}{1+y^2}\right) + \zeta(2q-1)\, \hyper{2}{1}\left(1,q\, ;\, \frac 32\, ;\, \frac{y^2}{1+y^2}\right) \right] + \\[0.3cm]
& + 4\, \left(\frac{y}{1+y^2}\right)^q \, \sum_{s=1}^\infty\, 2^s\, \frac{\Gamma(s+q)}{s!\, \Gamma(q)} \left( \sum_{k=1}^s\, {s\choose k-1}\, \frac{(-1)^k}{2^k\, (s+k)}\, B_{s+k}\, \zeta(2q-1+s+k) \right) \, \frac{1+y^{2s}}{(1+y^{2})^{s}}
\end{aligned}
\end{equation}

\noindent
We point out that in the previous expression there is no dependence on $\zeta(2q)$, hence there is no pole at $q=\frac 12$,  the only pole being at $q=1$ due to  the term $\zeta(2q-1)$.

\begin{theorem}\label{forma-eis-conv}
For any $q=\xi +i \eta$ with $\xi >0$, the series in \eqref{forma-eis} is uniformly convergent for $y$ in any compact interval in $(0,+\infty)$.
\end{theorem}

\begin{proof}
We show that it is enough to prove that for all $q=\xi +i \eta$ with $\xi >0$ it holds
\[
\sum_{k=1}^s\, {s\choose k-1}\, \frac{(-1)^k}{2^k\, (s+k)}\, B_{s+k}\, \zeta(2q-1+s+k) = O\left(\frac{s^6}{2^s}\right)\, .
\]
This estimate is proved in four steps. In the fifth and last step we conclude the proof of the theorem.

\vskip 0.2cm
\noindent \emph{Step 1}. We first show that for $s\ge 2$ and even
\begin{equation}\label{primo-passo-eis}
2\, \frac{(-1)^{\frac s2 -1}}{(2\pi\ell)^s}\, \sum_{\stackrel{k=1}{\text{$k$ even}}}^s\, {s\choose k-1}\, \frac{(-1)^{\frac k2} (s+k-1)!}{(4\pi \ell)^k} = \frac{s!}{2\, (2\pi)^s\, \ell^{s+\frac 12}}\, J_{s+\frac 12} (2\pi \ell)\, , \quad \forall\, \ell \ge 1
\end{equation}

For $k$ even, we can write $(-1)^{\frac k2}$ as $i^k$ and let $k'=k-1$, hence
\[
\sum_{\stackrel{k=1}{\text{$k$ even}}}^s\, {s\choose k-1}\, \frac{(-1)^{\frac k2} (s+k-1)!}{(4\pi \ell)^k} = \sum_{\stackrel{k=0}{\text{$k$ odd}}}^{s-1}\, {s\choose k}\, (s+k)! \left(\frac{i}{4\pi \ell}\right)^{k+1} =
\]
\[
= \frac 12\,  \frac{i}{4\pi \ell}\, \Big( \sum_{k=0}^{s}\, {s\choose k}\, (s+k)! \left(\frac{i}{4\pi \ell}\right)^{k} - \sum_{k=0}^{s}\, {s\choose k}\, (s+k)!\, (-1)^k\, \left(\frac{i}{4\pi \ell}\right)^{k} \Big) =
\]
\[
=  \frac 12\,  \frac{i}{4\pi \ell}\, \Big( \sum_{k=0}^{s}\, {s\choose k}\, (s+k)! \left(\frac{i}{4\pi \ell}\right)^{k} - \sum_{k=0}^{s}\, {s\choose k}\, (s+k)!\, \overline{\left(\frac{i}{4\pi \ell}\right)^{k}} \Big) =
\]
\[
= -\frac{1}{4\pi \ell}\, \Im \, \Big( \sum_{k=0}^{s}\, {s\choose k}\, (s+k)! \left(\frac{i}{4\pi \ell}\right)^{k} \Big)
\]
We now recall that using \cite[eq. 10.39.6, p. 255 and eq. 13.2.8, p. 322]{nist}, one can write
\[
\sum_{k=0}^{s}\, {s\choose k}\, (s+k)!\, z^{k} = \frac{s!\, e^{\frac{1}{2z}}}{\sqrt{\pi z}}\, K_{s+\frac 12} \left(\frac{1}{2z} \right)
\]
hence
\begin{equation} \label{interm-1}
2\, \frac{(-1)^{\frac s2 -1}}{(2\pi\ell)^s}\, \sum_{\stackrel{k=1}{\text{$k$ even}}}^s\, {s\choose k-1}\, \frac{(-1)^{\frac k2} (s+k-1)!}{(4\pi \ell)^k} = \frac{2\, (-1)^{\frac s2}\, s!}{(2\pi)^{s+1}\, \ell^{s+\frac 12}}\, \Im \, \Big( e^{-i\, \frac \pi 4}\, K_{s+\frac 12} (-2\pi i\ell)\Big)
\end{equation}
Finally, we use \cite[eq. 10.27.8, p. 251]{nist}, relating the modified Bessel function $K_\nu$ to the Hankel function $H^{(1)}_\nu$, to write
\[
e^{-i\, \frac \pi 4}\, K_{s+\frac 12} (-2\pi i\ell ) = \frac{i\, (-1)^{\frac s2}\, \pi}{2}\, H^{(1)}_{s+\frac 12} (2\pi\ell)\, ,
\]
and the relation
\[
H^{(1)}_\nu (z) = J_\nu(z) + i\, Y_\nu(z)\, ,
\]
where the Bessel functions of first and second kind $J_\nu$ and $Y_\nu$ are real for real orders $\nu$ and positive real argument $z$, to obtain
\[
\Im \, \Big( e^{-i\, \frac \pi 4}\, K_{s+\frac 12} (-2\pi i)\Big) = \frac{(-1)^{\frac s2}\, \pi}{2}\, J_{s+\frac 12}(2\pi \ell)\, .
\]
Using it in \eqref{interm-1}, we obtain \eqref{primo-passo-eis}.

\vskip 0.2cm
\noindent \emph{Step 2}. Following the same ideas as in Step 1, we show that for $s\ge 1$ and odd
\begin{equation}\label{primo-passo-eis-odd}
2\, \frac{(-1)^{\frac{s+1}{2} -1}}{(2\pi\ell)^s}\, \sum_{\stackrel{k=1}{\text{$k$ odd}}}^s\, {s\choose k-1}\, \frac{(-1)^{\frac{k+1}{2}} (s+k-1)!}{(4\pi\ell)^k} = \frac{s!}{2\, (2\pi)^s\, \ell^{s+\frac 12}}\, J_{s+\frac 12} (2\pi \ell) \, , \quad \forall\, \ell \ge 1 
\end{equation}

\vskip 0.2cm
\noindent \emph{Step 3}. For all $s\ge1$ we have
\begin{equation}\label{penultimo-passo-eis}
\Big|\, \sum_{k=1}^s\, {s\choose k-1}\, \frac{(-1)^k}{(2m)^k\, (s+k)}\, B_{s+k} \Big| \le \frac{4\, s^2\, m^{s}}{2^s}\, , \quad \forall\, m\ge 1
\end{equation}

\noindent
We recall that Bernoulli numbers satisfy $B_n=0$ for $n\ge 2$ and odd, and the identity
\[
B_{2n} = \frac{(-1)^{n-1}\, 2\, (2n)!}{(2\pi)^{2n}}\, \zeta(2n)\, , \quad \forall\, n\ge 1\, .
\]
Hence, we write the left hand side of \eqref{penultimo-passo-eis} for $s$ even, for which the only non-vanishing terms correspond to $k$ even, as
\begin{equation}\label{penultimo-even}
2\, \frac{(-1)^{\frac s2 -1}}{(2\pi)^s}\, \sum_{\stackrel{k=1}{\text{$k$ even}}}^s\, {s\choose k-1}\, \frac{(-1)^{\frac k2} (s+k-1)!}{(4\pi m)^k}\, \zeta(s+k)
\end{equation}
and for $s$ odd, for which the only non-vanishing terms correspond to $k$ odd, as
\begin{equation}\label{penultimo-odd}
2\, \frac{(-1)^{\frac{s+1}{2} -1}}{(2\pi)^s}\, \sum_{\stackrel{k=1}{\text{$k$ odd}}}^s\, {s\choose k-1}\, \frac{(-1)^{\frac{k+1}{2}} (s+k-1)!}{(4\pi m)^k}\, \zeta(s+k)
\end{equation}
In both cases, using the Euler-MacLaurin formula for the Riemann zeta function, we can write for all fixed $N\ge 1$
\[
\zeta(s+k) = \sum_{n=1}^N\, \frac{1}{n^{s+k}} + R_{s+k}(N)
\]
where $|R_{s+k}(N)| \le 2 \max\{1, \frac{N}{s+k-1}\}\, N^{-s-k}$.

\noindent
Let us first consider the case $s$ even. From \eqref{penultimo-even} we are reduced to study the sum
\begin{equation}\label{even-q-fin}
\sum_{n=1}^N\, 2\, \frac{(-1)^{\frac s2 -1}}{(2\pi n)^s}\, \sum_{\stackrel{k=1}{\text{$k$ even}}}^s\, {s\choose k-1}\, \frac{(-1)^{\frac k2} (s+k-1)!}{(4\pi m n)^k} + 2\, \frac{(-1)^{\frac s2 -1}}{(2\pi)^s}\, \sum_{\stackrel{k=1}{\text{$k$ even}}}^s\, {s\choose k-1}\, \frac{(-1)^{\frac k2} (s+k-1)!}{(4\pi m)^k}\, R_{s+k}(N)
\end{equation}
For the first term we apply Step 1, and applying \eqref{primo-passo-eis} with $\ell = mn$ we find
\[
m^s\, \sum_{n=1}^N\, 2\, \frac{(-1)^{\frac s2 -1}}{(2\pi m n)^s}\, \sum_{\stackrel{k=1}{\text{$k$ even}}}^s\, {s\choose k-1}\, \frac{(-1)^{\frac k2} (s+k-1)!}{(4\pi m n)^k} = \frac{m^s\, s!}{2\, (2\pi)^s} \sum_{n=1}^N\, \frac{1}{(mn)^{s+\frac 12}}\, J_{s+\frac 12} (2\pi mn)
\]
Since
\[
\left| J_{s+\frac 12} (2\pi \ell) \right| \le \frac{(\pi \ell)^{s+\frac 12}}{\Gamma(s+1)}
\]
we get
\begin{equation} \label{stimiamo-1-somma}
\Big| m^s\, \sum_{n=1}^N\, 2\, \frac{(-1)^{\frac s2 -1}}{(2\pi m n)^s}\, \sum_{\stackrel{k=1}{\text{$k$ even}}}^s\, {s\choose k-1}\, \frac{(-1)^{\frac k2} (s+k-1)!}{(4\pi m n)^k} \Big| \le \frac{\pi^{\frac 12}\, N\, m^s}{2^{s+1}}\, , \quad \forall\, N\ge 1\, .
\end{equation}
For the second term of \eqref{even-q-fin}, we apply the crude estimate
\[
\Big| 2\, \frac{(-1)^{\frac s2 -1}}{(2\pi)^s}\, \sum_{\stackrel{k=1}{\text{$k$ even}}}^s\, {s\choose k-1}\, \frac{(-1)^{\frac k2} (s+k-1)!}{(4\pi m)^k}\, R_{s+k}(N) \Big| \le \frac{4\, (2s)! \, \max\{1, \frac{N}{s} \}}{(2\pi N)^s}\, \sum_{k=1}^s\, {s\choose k-1}\, \frac{1}{(4\pi m N)^k} =
\]
\[
= \frac{2\, (2s)! \, \max\{1, \frac{N}{s} \}}{(2\pi N)^{s+1}\, m}\, \left(1+\frac{1}{4\pi m N}\right)^s
\]
From this and \eqref{stimiamo-1-somma} it follows that for $s$ even, we obtain the estimate
\begin{equation}\label{ogni-n-even}
\Big| 2\, \frac{(-1)^{\frac s2 -1}}{(2\pi)^s}\, \sum_{\stackrel{k=1}{\text{$k$ even}}}^s\, {s\choose k-1}\, \frac{(-1)^{\frac k2} (s+k-1)!}{(4\pi m)^k}\, \zeta(s+k) \Big| \le \frac{\pi^{\frac 12}\, N\, m^s}{2^{s+1}} + \frac{2\, (2s)! \, \max\{1, \frac{N}{s} \}}{(2\pi N)^{s+1}\, m}\, \left(1+\frac{1}{4\pi m N}\right)^s
\end{equation}
which holds for all $N\ge 1$. Hence, choosing $N=2\, s^2$ in \eqref{ogni-n-even} and applying standard estimates for the factorial term we obtain
\[
\Big| 2\, \frac{(-1)^{\frac s2 -1}}{(2\pi)^s}\, \sum_{\stackrel{k=1}{\text{$k$ even}}}^s\, {s\choose k-1}\, \frac{(-1)^{\frac k2} (s+k-1)!}{(4\pi)^k}\, \zeta(s+k) \Big| \le \frac{4\, s^2\, m^s}{2^{s}} 
\]
The case $s$ odd follows exactly by the same argument using \eqref{primo-passo-eis-odd} in \eqref{penultimo-odd}.

\vskip 0.2cm
\noindent \emph{Step 4}. For all $q=\xi+i\eta$ with $\xi>0$, we have
\begin{equation}\label{ultimo-passo-eis}
\sum_{k=1}^s\, {s\choose k-1}\, \frac{(-1)^k}{2^k\, (s+k)}\, B_{s+k}\, \zeta(2q-1+s+k) = O\left(\frac{s^6}{2^s}\right)\, .
\end{equation}

In Step 3 we have an estimate for the left hand side with $\zeta(2q-1+s+k)=1$. Using the formula
\[
\zeta(2q-1+s+k) = \sum_{m=1}^M\, \frac{1}{m^{2q-1+s+k}} + R_{2q-1+s+k}(M)\, , \quad \forall\, M\ge 1
\]
where $|R_{2q-1+s+k}(M)| \le 2 \max\{1, \frac{M}{|2q-2+s+k|}\}\, M^{-2\xi +1-s-k}$, we have to estimate the sum
\[
\sum_{m=1}^M\, \frac{1}{m^{2q-1+s}}\, \sum_{k=1}^s\, {s\choose k-1}\, \frac{(-1)^k}{(2m)^k\, (s+k)}\, B_{s+k} + \sum_{k=1}^s\, {s\choose k-1}\, \frac{(-1)^k}{2^k\, (s+k)}\, B_{s+k}\, R_{2q-1+s+k}(M)
\]
and we can use Step 3. For the first term we apply \eqref{penultimo-passo-eis} to write
\[
\Big|\, \sum_{m=1}^M\, \frac{1}{m^{2q-1+s}}\, \sum_{k=1}^s\, {s\choose k-1}\, \frac{(-1)^k}{(2m)^k\, (s+k)}\, B_{s+k} \, \Big| \le \frac{4\, s^2}{2^s}\, \sum_{m=1}^M\, \frac{1}{m^{2\xi-1}}
\]
For the second term, we use the classical estimate
\[
|B_{2n}| \le 5\, \sqrt{\pi n}\, \left(\frac{n}{\pi e}\right)^{2n}
\]
and write the crude estimate
\[
\Big|\, \sum_{k=1}^s\, {s\choose k-1}\, \frac{(-1)^k}{2^k\, (s+k)}\, B_{s+k}\, R_{2q-1+s+k}(M) \, \Big| \le \frac{5\, \sqrt{2\pi}\, \max\{1, \frac{M}{s-2}\}}{M^{2\xi+s-1}}\, \sum_{k=1}^s\, {s\choose k-1}\, \frac{(s+k)^{s+k}}{(2M)^k\, (2\pi e)^{s+k}\, \sqrt{s+k}} \le
\] 
\[
\le \frac{5\, \sqrt{2\pi}\, \max\{1, \frac{M}{s-2}\}}{M^{2\xi+s-1}}\, \frac{(2s)^{2s}}{(2\pi e)^s}\, \sum_{k=1}^s\, {s\choose k-1}\, \frac{1}{(4\pi e M)^k} = \frac{5\, \sqrt{2\pi}\, \max\{1, \frac{M}{s-2}\}\, (2s)^{2s}}{M^{2\xi+s-1}\, (2\pi e)^s}\, \left(1+\frac{1}{4\pi e M}\right)^s
\]
We have thus obtained that for $s\ge 3$ and all $q=\xi +i\eta$ with $\xi >0$, the estimate
\[
\Big|\, \sum_{k=1}^s\, {s\choose k-1}\, \frac{(-1)^k}{2^k\, (s+k)}\, B_{s+k}\, \zeta(2q-1+s+k) \Big| \le \frac{4\, s^2}{2^s}\, \sum_{m=1}^M\, \frac{1}{m^{2\xi-1}} + \frac{5\, \sqrt{2\pi}\, \max\{1, \frac{M}{s-2}\}\, s^{2s}}{M^{2\xi+s-1}}\, \left(1+\frac{1}{4\pi e M}\right)^s
\]
holds for all $M\ge 1$. Choosing $M=2s^2$ the estimate \eqref{ultimo-passo-eis} follows.

\vskip 0.2cm
\noindent \emph{Step 5}. The proof is finished by using \eqref{ultimo-passo-eis} and the crude estimate
\[
\frac{\Gamma(s+q)}{s!} = O(|s+q|^\xi)\, ,
\]
since for all $a,b>0$  
\[
\sup_{y\in [a,b]}\, \frac{1+y^{2s}}{(1+y^{2})^{s}} \le \frac{1}{(1+a^{2})^{s}} + \frac{b^{2s}}{(1+b^{2})^{s}}\, .
\]
\end{proof}

\appendix

\section{Proof of \eqref{vero}} \label{app1}

Using the notation $\lfloor x \rfloor$ and $\{ x\}$ for the integer and fractional part of a real number, we have
$$
S_{k}(n) - \frac{1}{k+1}\, n^{k+1} - \frac 12\, n^{k} = \int_{0}^{n}\, \left[ (\lfloor x \rfloor +1)^{k} - x^{k} - \frac 12\, k x^{k-1} \right]\, dx =
$$
$$
= \int_{0}^{n}\, \left[ \sum_{j=0}^{k}\, \left( \begin{array}{c} k\\ j \end{array} \right)\, \lfloor x \rfloor ^{j} - \sum_{j=0}^{k}\, \left( \begin{array}{c} k\\ j \end{array} \right)\, \lfloor x \rfloor ^{j}\,  \{ x\}^{k-j} - \frac 12\, k \sum_{j=0}^{k-1}\, \left( \begin{array}{c} k-1\\ j \end{array} \right)\, \lfloor x \rfloor ^{j}\,  \{ x\}^{k-j-1} \right]\, dx =
$$
$$
= \int_{0}^{n}\, \Big[ \lfloor x \rfloor ^{k} \left( 1-\{ x\}^{k-k} \right) + \lfloor x \rfloor ^{k-1} \left( k - k \{ x\} - \frac 12\, k \right) +
$$
$$ + \sum_{j=0}^{k-2}\, \lfloor x \rfloor ^{j}\, \left( \left( \begin{array}{c} k\\ j \end{array} \right)\, (1- \{ x\}^{k-j}) - \frac 12\, k\, \left( \begin{array}{c} k-1\\ j \end{array} \right)\, \{ x\}^{k-j-1} \right) \Big]\, dx =
$$
$$
= \sum_{h=0}^{n-1}\, \int_{h}^{h+1}\, k\, \lfloor x \rfloor ^{k-1} \left( \frac 12 - \{ x\} \right)\, dx + \int_{0}^{n}\, \left[ \sum_{j=0}^{k-2}\, \left( \begin{array}{c} k\\ j \end{array} \right)\, \lfloor x \rfloor ^{j}\, \left( 1 - \{ x\}^{k-j} - \frac{k-j}{2}\, \{ x\}^{k-j-1} \right)\, \right] \, dx =
$$
$$
= \sum_{h=0}^{n-1}\, \int_{h}^{h+1}\, k\, h^{k-1} \left( \frac 12 - x+h \right) \, dx + \int_{0}^{n}\, \left[ \sum_{j=0}^{k-2}\, \left( \begin{array}{c} k\\ j \end{array} \right)\, \lfloor x \rfloor ^{j}\, \left( 1 - \{ x\}^{k-j} - \frac{k-j}{2}\, \{ x\}^{k-j-1} \right)\, \right] \, dx =
$$
$$
= \int_{0}^{n}\, \left[ \sum_{j=0}^{k-2}\, \left( \begin{array}{c} k\\ j \end{array} \right)\, \lfloor x \rfloor ^{j}\, \left( 1 - \{ x\}^{k-j} - \frac{k-j}{2}\, \{ x\}^{k-j-1} \right)\, \right] \, dx
$$
We can then write
\[
\Big| S_{k}(n) - \frac{1}{k+1}\, n^{k+1} - \frac 12\, n^{k} \Big| \le k\,  \int_{0}^{n}\, \sum_{j=0}^{k-2}\, \left( \begin{array}{c} k\\ j \end{array} \right)\, \lfloor x \rfloor ^{j}\, dx =
$$
$$ = \int_{0}^{n}\,  k^2(k-1) \sum_{j=0}^{k-2}\, \frac{1}{(k-j)(k-j-1)}\, \left( \begin{array}{c} k-2\\ j \end{array} \right)\, \lfloor x \rfloor ^{j}\, dx \le
\]
$$
\le k^3\, \int_{0}^{n}\, \sum_{j=0}^{k-2}\, \left( \begin{array}{c} k-2\\ j \end{array} \right)\, \lfloor x \rfloor ^{j}\, dx = k^3\, \int_{0}^{n}\, (\lfloor x \rfloor +1)^{k-2}\, dx \le \text{const}\, k^2\, n^{k-1}
$$

\section{Spectral properties of the terms with Legendre functions} \label{leg-spec}

We now give some properties of the functions used in the series \eqref{legendre-exp} in terms of the hyperbolic Laplacian $\Delta = - y^{2}\, \left( \frac{\partial^{2}}{\partial x^{2}} + \frac{\partial^{2}}{\partial y^{2}} \right)$. Using recurrence relations and the formulas for derivatives of the Legendre functions which can be found in \cite[vol. I]{E}, with the notation
\[
F_{q}(n,y):= \frac{y^{n+q}}{(1+y^{2})^{\frac n2 + \frac q2 + \frac 14}} \, \PPP^{-q+\frac 12}_{n+q-\frac 12} \left( \frac{y}{(1+y^{2})^{\frac 12}} \right)
\]
we find that for all $n\ge 0$
\[
\begin{array}{c}
-y^{2}\, \frac{\partial^{2}}{\partial y^{2}} \left[ F_{q}(n,y) \right] = \\[0.3cm]
= -(n+q)(n+q-1)\, F_{q}(n,y)  + 2(n+2q)(n+q)\, F_{q}(n+1,y) - (n+2q)(n+2q+1)\, F_{q}(n+2,y)\, .
\end{array}
\]
Hence for a series 
$$
\FFF_{q}(y) := \sum_{n=0}^{\infty}\, (-1)^{n} a_{n} F_{q}(n,y)\, , \quad y\in (0,\infty)
$$
with $\limsup |a_{n}|^{\frac 1n}\le 1$, we find
\begin{equation}\label{second-der}
-y^{2}\, \frac{\partial^{2}}{\partial y^{2}} \FFF_{q}(y) = \sum_{n=0}^{\infty}\, (-1)^{n} b_{n}\, F_{q}(n,y)
\end{equation}
where
\begin{equation}\label{system-1}
\left\{
\begin{array}{l}
b_{0} = q(1-q)\, a_{0}\\[0.2cm]
b_{1} = -q(1+q)\, a_{1} - 4q^{2}\, a_{0}\\[0.2cm]
b_{n} = -(n+q)(n+q-1) a_{n} -2(n+2q-1)(n+q-1) a_{n-1} - (n+2q-2)(n+2q-1) a_{n-2}, \quad n\ge 2
\end{array}
\right.
\end{equation}
It follows that
\begin{theorem}\label{serie-eis-x0}
For any $q$ the function
$$
\EEE_{q}(y) = \sum_{n=0}^{\infty}\, \frac{\Gamma(n+2q)}{n!\, \Gamma(2q)}\, F_{q}(n,y)
$$
satisfies $\Delta \EEE_{q}(y) = q(1-q) \EEE_{q}(y)$.
\end{theorem}

\begin{proof}
Using \eqref{second-der} and \eqref{system-1} we have to find a solution of the system
\begin{equation}\label{system-2}
\left\{
\begin{array}{l}
a_{0} \in \C \\[0.2cm]
a_{1} = -2q\, a_{0}\\[0.2cm]
a_{n} = -2\, \frac{n+q-1}{n}\, a_{n-1} -\frac{n+2q-2}{n}\, a_{n-2}, \qquad n\ge 2
\end{array}
\right.
\end{equation}
with $\limsup |a_{n}|^{\frac 1n}\le 1$.

\noindent
Consider the generating function
$$
f(z) = \sum_{n\ge 0}\, a_{n}z^{n}
$$
with $a_{0}=1$, of the solution of \eqref{system-2}. From the recurrence relation of $(a_{n})$, it turns out that $f$ satisfies
\begin{equation} \label{equaz-f}
\left\{
\begin{array}{l}
(1+z)^{2}\, f'(z) + 2q(z+1)\, f(z) = 0 \\[0.3cm]
f(0)=1
\end{array}
\right.
\end{equation}
We now want to show that the solution $f$ of \eqref{equaz-f} is analytic for $|z|<1$. Letting for any $\alpha \in \C$
$$
(1+z)^{\alpha}:= \exp\Big( \alpha \log|1+z| + i\alpha \text{arg}(1+z)\Big)
$$
with $\text{arg}(1+z) \in (-\pi,\pi]$, we have that $(1+z)^{\alpha}$ is well defined as a single-valued analytic function on the cut plane $\C \setminus (-\infty,-1]$, hence in particular for $|z|<1$. It follows that
$$
f(z) = (1+z)^{-2q}
$$
is the solution of \eqref{equaz-f} and is analytic for $|z|<1$. Hence for all $n\ge 0$
$$
a_{n} = (-1)^{n}\, \frac{\Gamma(n+2q)}{n!\, \Gamma(2q)}
$$
is the solution of \eqref{system-2} with $a_{0}=1$ and $\limsup |a_{n}|^{\frac 1n}\le 1$.
\end{proof}

\noindent
It is interesting to notice that the functions $\EEE_{q}(y)$ have a much simpler formulation.

\begin{theorem}\label{eisen-ser}
For all $q$ we have
$$
\EEE_{q}(y) = \frac{2^{-q+\frac 12}}{\Gamma\left(q+\frac 12\right)} \, y^{q}\, ,
$$
hence for $\Re(q)>1$ it holds
$$
\EEE_{q}(z) := \sum_{(c,d)=1}\, \EEE_{q}\left( \frac{y}{|cz+d|^{2}} \right) = \mathrm{const.(q)}\ E(z,q)
$$
where $E(z,q)$ is the non-holomorphic Einsenstein series.
\end{theorem}

\begin{proof}
We start from the expression for $\EEE_q$ found in Theorem \ref{serie-eis-x0}. Using the relation \cite[eq. 14.3.21, p. 355]{nist}
\[
\PPP^\mu_\nu(t) = \frac{2^\mu\, \Gamma(1-2\mu)\, \Gamma(\nu+\mu+1)}{\Gamma(\nu-\mu+1)\, \Gamma(1-\mu)\, (1-t^2)^{\frac \mu 2}}\, C^{(\frac 12-\mu)}_{\nu+\mu} (t)
\]
in terms of the Gegenbauer functions $C_\alpha^{(\beta)}$, we get
\[
\EEE_{q}(y) = \frac{2^{-q+\frac 12}}{\Gamma\left(q+\frac 12\right)}\, \left( \frac{y}{1+y^2} \right)^q\, \sum_{n=0}^\infty\, \left(\frac{y}{(1+y^2)^{\frac 12}}\right)^n\, C^{(q)}_n \left(\frac{y}{(1+y^2)^{\frac 12}}\right)
\]
Finally, the generating function \cite[eq. 18.12.4, p. 449]{nist}
\[
(1-2t\, u+u^2)^{-\beta} = \sum_{n=0}^\infty\, u^n\, C^{(\beta)}_\alpha(t)
\]
valid for $|u|<1$, we conclude that
\[
\EEE_{q}(y) = \frac{2^{-q+\frac 12}}{\Gamma\left(q+\frac 12\right)} \, y^{q}\, .
\]
The last statement is immediate from the definition of the non-holomorphic Eisenstein series.
\end{proof}

\section*{Acknowledgements.}
We thank Kyle Butler-Moore for suggesting to use Gegenbauer functions in the proof of Theorem \ref{eisen-ser}, and the anonymous referee for very interesting comments. The authors are partially supported by ``Gruppo Nazionale per la Fisica Matematica'' (GNFM) of Istituto Nazionale di Alta Matematica (INdAM), and by the research project PRIN 2017S35EHN$\textunderscore$004 ``Regular and stochastic behaviour in dynamical systems'' of the Italian Ministry of Education and Research. C. Bonanno was also partially supported by the research project PRA$\textunderscore$2017$\textunderscore$22 ``Dynamical systems in analysis, geometry, mathematical logic and celestial mechanics'' of the University of Pisa. This work is part of the authors' research within the DinAmicI community, see www.dinamici.org


\end{document}